\documentclass[a4paper,11pt]{scrartcl}
\setkomafont{section}{\normalfont\bfseries\large\selectfont}
\setkomafont{subsection}{\normalfont\bfseries\large\selectfont}
\usepackage{mmap}%
\usepackage{microtype}
\usepackage[utf8]{inputenc}
\usepackage[T1]{fontenc}
\usepackage{paralist}
\usepackage{subcaption}
\usepackage{amsmath,amssymb,amsthm}
\usepackage{multirow}
\usepackage{etoolbox}
\usepackage{graphicx}
\usepackage{tikz}
\usepackage[margin=1in]{geometry}
\usetikzlibrary{shapes.geometric}
\usetikzlibrary{positioning,arrows.meta,patterns,ipe}
\usepackage{titling}

\usepackage[%
    backend=biber,%
    style=ieee,%
    sortcites,%
    dashed=false,%
    citestyle=numeric,%
    sorting=anyt,%
    giveninits=true,%
    url=false,%
    doi=true,%
    isbn=false,%
    eprint=false,%
    mincitenames=1,%
    maxcitenames=3,%
    ]{biblatex}
\AtEveryBibitem{%
    \clearname{editor}%
    \clearfield{url}%
    \clearfield{urldate}%
    \clearfield{review}%
    \clearfield{series}%
    \clearfield{issn}%
    \clearfield{isbn}%
}

\usepackage{fancyhdr}

\usepackage{hyperref}
\hypersetup{
	colorlinks=true,
    linkcolor={red!50!black},
    citecolor={blue!50!black},
    urlcolor={blue!80!black},
    bookmarksopen=true,
	  bookmarksnumbered,
	  bookmarksopenlevel=2,
	  bookmarksdepth=3
}
\usepackage[nameinlink,capitalise]{cleveref}
\usepackage[shortcuts]{extdash}

\makeatletter\if@cref@capitalise
\crefname{claim}{Claim}{Claims}
\crefname{condition}{Condition}{Conditions}
\crefname{construction}{Construction}{Constructions}
\crefname{corollary}{Corollary}{Corollaries}
\crefname{definition}{Definition}{Definitions}
\crefname{example}{Example}{Examples}
\crefname{exercise}{Exercise}{Exercises}
\crefname{lemma}{Lemma}{Lemmas}
\crefname{observation}{Observation}{Observations}
\crefname{proposition}{Proposition}{Propositions}
\crefname{question}{Question}{Questions}
\crefname{remark}{Remark}{Remarks}
\crefname{theorem}{Theorem}{Theorems}
\else
\crefname{claim}{claim}{claims}
\crefname{condition}{condition}{conditions}
\crefname{construction}{claim}{constructions}
\crefname{corollary}{corollary}{corollaries}
\crefname{definition}{definition}{definitions}
\crefname{example}{example}{examples}
\crefname{exercise}{exercise}{exercises}
\crefname{lemma}{lemma}{lemmas}
\crefname{observation}{observation}{observations}
\crefname{proposition}{proposition}{propositions}
\crefname{question}{question}{questions}
\crefname{remark}{remark}{remarks}
\crefname{theorem}{theorem}{theorems}
\fi
\Crefname{claim}{Claim}{Claims}
\Crefname{condition}{Condition}{Conditions}
\Crefname{construction}{Construction}{Constructions}
\Crefname{corollary}{Corollary}{Corollaries}
\Crefname{definition}{Definition}{Definitions}
\Crefname{example}{Example}{Examples}
\Crefname{exercise}{Exercise}{Exercises}
\Crefname{lemma}{Lemma}{Lemmas}
\crefname{observation}{Observation}{Observations}
\Crefname{proposition}{Proposition}{Propositions}
\Crefname{question}{Question}{Questions}
\Crefname{remark}{Remark}{Remarks}
\Crefname{theorem}{Theorem}{Theorems}
\makeatother

\counterwithin{figure}{section}

\newtheorem{theorem}{Theorem}[section]
\newtheorem{lemma}[theorem]{Lemma}
\newtheorem{corollary}[theorem]{Corollary}
\newtheorem{observation}[theorem]{Observation}

\newtheorem{remark}[theorem]{Remark}
\newtheorem{question}[theorem]{Question}

\newcommand{\tw}{\textrm{tw}}
\newcommand{\h}{\eta}
\newcommand{\G}{\mathcal{G}}
\newcommand{\N}{\mathbb{N}}
\renewcommand{\P}{\textsf{P}}
\newcommand{\NP}{\textsf{NP}}
\newcommand{\APX}{\textsf{APX}}
\usepackage{ textcomp }
\newcommand{\mystar}{\text{\textasteriskcentered}}

\title{Treewidth versus clique number. I.\\Graph classes with a forbidden structure%
\thanks{An extended abstract of this work appeared in the proceedings of the 46th International Workshop on Graph-Theoretic Concepts in Computer Science (WG 2020)~\cite{DBLP:conf/wg/DallardMS20}.}~
\thanks{This research was funded in part by the Slovenian Research Agency (I0-0035, research program P1-0285, research projects J1-9110, N1-0102, and N1-0160, and a Young Researchers Grant).}}

\preauthor{}
\DeclareRobustCommand{\authorthing}{
\begin{center}
Cl\'{e}ment Dallard\textsuperscript{1} \hspace{20pt} Martin Milani\v{c}\textsuperscript{1} \hspace{20pt} Kenny \v{S}torgel\textsuperscript{2}\\
\medskip
\textsuperscript{1}~FAMNIT and IAM, University of Primorska, Koper, Slovenia\\
\textsuperscript{2}~Faculty of Information Studies, Novo mesto, Slovenia\\
\smallskip
\href{mailto:clement.dallard@famnit.upr.si}{clement.dallard@famnit.upr.si}\\
\href{mailto:milanic.martin@upr.si}{milanic.martin@upr.si}\\
\href{mailto:kennystorgel.research@gmail.com}{kennystorgel.research@gmail.com}

\end{center}
}
\author{\authorthing}
\postauthor{}

\predate{}
\date{}
\postdate{}

\begin{document}
\maketitle

\begin{abstract}
Treewidth is an important graph invariant, relevant for both structural and algorithmic reasons. A necessary condition for a graph class to have bounded treewidth is the absence of large cliques. We study graph classes closed under taking induced subgraphs in which this condition is also sufficient, which we call $(\tw,\omega)$-bounded. Such graph classes are known to have useful algorithmic applications related to variants of the clique and $k$-coloring problems.
We consider six well-known graph containment relations: the minor,
topological minor,
subgraph,
induced minor,
induced topological minor, and
induced subgraph relations. For each of them, we give a complete characterization of the graphs $H$ for which the class of graphs excluding $H$ is $(\tw,\omega)$-bounded.
Our results yield an infinite family of $\chi$-bounded induced-minor-closed graph classes and imply that the class of $1$-perfectly orientable graphs is $(\tw,\omega)$-bounded, leading to linear-time algorithms for $k$-coloring $1$-perfectly orientable graphs for every fixed~$k$. This answers a question of Bre\v sar, Hartinger, Kos, and Milani{\v c} from 2018 and one of Beisegel, Chudnovsky, Gurvich, Milani{\v c}, and Servatius from 2019, respectively. We also reveal some further algorithmic implications of $(\tw,\omega)$-boundedness related to list $k$-coloring and clique problems. In addition, we propose a question about the complexity of the maximum weight independent set problem in  $(\tw,\omega)$-bounded graph classes and prove that the problem is polynomial-time solvable in every class of graphs excluding a fixed star as an induced minor.
\end{abstract}

\section{Introduction}

\subsection{Background and motivation}

The treewidth of a graph measures, roughly speaking, how similar the graph is to a tree. This invariant played a crucial role in the theory of graph minors due to Robertson and Seymour (see, e.g.,~\cite{MR2188176}), and many decision and optimization problems that are generally \NP-hard are solvable in linear time for graph classes of bounded treewidth~\cite{MR1154588,MR1417901,MR1042649}. A necessary condition for bounded treewidth is the absence of large cliques. When is this condition also sufficient?
We say that a graph class $\mathcal{G}$ is \emph{$(\tw,\omega)$-bounded} if there exists a function $f:\N\to \N$ such that for every graph $G\in \G$ and every induced subgraph $G'$ of $G$, we have $\tw(G')\le f(\omega(G'))$, where $\tw(G')$ and $\omega(G')$ denote the treewidth and the clique number of $G'$, respectively.
Such a function $f$ is called a \emph{$(\tw,\omega)$-binding function} for the class $\mathcal{G}$. Many graph classes studied in the literature are known to be $(\tw,\omega)$-bounded. For every positive integer $t$, the class of intersection graphs of connected subgraphs of graphs with treewidth at most~$t$ is $(\tw,\omega)$-bounded \cite{MR1090614,MR1642971}. This includes the classes of chordal graphs and circular-arc graphs.
Further examples include graph classes of bounded treewidth, classes of graphs in which all minimal separators are of bounded size~\cite{MR1852483}, and, as a consequence of Ramsey's theorem, classes of graphs of bounded independence number.

There are multiple motivations for the study of $(\tw,\omega)$-bounded graph classes, from both algorithmic and structural points of view. The \textsc{$k$-Clique} problem asks whether the input graph contains a clique of size $k$; the problem is known to be \textsf{W[1]}-hard (see, e.g.,~\cite{MR1323150}). Given a graph $G$ and a list of available colors from the set $\{1,\ldots, k\}$ for each vertex, the \textsc{List $k$-Coloring} problem asks whether $G$ can be properly vertex-colored by assigning to each vertex a color from its list. This is a generalization of the classical $k$-coloring problem and is thus \NP-hard for all $k\ge 3$ (see, e.g.,~\cite{MR2048545,MR3623382}).
Chaplick and Zeman gave fixed-parameter tractable algorithms for \textsc{$k$-Clique} and \textsc{List $k$-Coloring} in any $(\tw,\omega)$-bounded class of graphs with a computable binding function $f$~\cite{DBLP:journals/endm/ChaplickZ17}.
For a fixed value of $k$, their approach leads to a linear-time algorithm for the \textsc{$k$-Clique} and \textsc{List $k$-Coloring} problems in any such graph class.\footnote{In fact, they consider a more general setting where the inequality $\tw(G)\le f(\omega(G))$, for a computable function~$f$, is only required to hold for the input graph $G$ and not necessarily for all its induced subgraphs.\label{footnote non-hereditary}}
From the structural point of view, identifying new $(\tw,\omega)$-bounded graph classes directly addresses a recent question of Wei{\ss}auer~\cite{MR4010195} asking for which classes we can force large cliques by assuming large treewidth.  Wei{\ss}auer distinguishes graph parameters as being either \emph{global} or \emph{local} (see~\cite{MR4010195} for precise definitions). In this terminology, $(\tw,\omega)$-boundedness of a graph class is a sufficient condition for treewidth to become a local parameter.

\subsection{Our results}

The main aim of this paper is to further the knowledge of $(\tw,\omega)$-bounded graph classes.
We consider six well-known graph containment relations and for each of them give a complete characterization of the graphs $H$ for which the class of graphs excluding $H$ (with respect to the relation) is $(\tw,\omega)$-bounded.
These six relations are the minor relation, the topological minor relation, the subgraph relation,
and their induced variants, the induced minor relation, the induced topological minor relation, and the induced subgraph relation. (Precise definitions will be given in \cref{sec:prelim}.)
To explain our results, we need to introduce some notation. We denote by $\subseteq_{is}$ the induced subgraph relation.
By $K_{p,q}$ we denote the complete bipartite graph with parts of size $p$ and $q$; if $p = q$, then the complete bipartite graph is said to be \emph{balanced}. The \emph{claw} is the complete bipartite graph $K_{1,3}$. A \emph{subdivided claw} is the graph obtained from the claw by replacing each edge with a path of length at least one.
We denote by $\mathcal{S}$ the class of graphs in which every connected component is either a path or a subdivided claw.
For $q\ge 1$, we denote by $K_{2,q}^+$ the graph obtained from $K_{2,q}$ by adding an additional edge between the two vertices in the part of size $2$. Similarly, we denote by $K_q^{-}$ the graph obtained from the complete graph $K_q$ by removing an edge.
Note that the graph $K_4^-$ is sometimes called the \emph{diamond}. The graph $C_\ell$ is the cycle on $\ell$ vertices, and the \emph{$4$-wheel}, also denoted by $W_4$, is the graph obtained from the $C_4$ by adding a new vertex adjacent to all vertices of the $C_4$. A graph is \emph{subcubic} if every vertex is incident with at most three edges.

Our characterizations are summarized in~\cref{table-results} where each entry corresponds to one of the six containment relations and contains a description of necessary and sufficient conditions for a graph $H$ such that the class of graphs excluding $H$ with respect to the relation considered in the entry is $(\tw,\omega)$-bounded. When forbidding $H$ as a subgraph, a topological minor, or a minor, $(\tw,\omega)$-boundedness turns out to be equivalent to boundedness of the treewidth. However, this is not the case for the induced variants.

\begin{table}[htb]
	\renewcommand{\arraystretch}{1.25}
	\setlength{\tabcolsep}{7pt}
	\centering
    \begin{tabular}{|l||c|c|}
		\hline
		& General & Induced\\
		\hline\hline
		Subgraph & $H\in \mathcal{S}$ & $H\subseteq_{is} P_3$ or $H$ is edgeless\\
		\hline
		\multirow{2}{*}{Topological minor} & $H$ is subcubic
		& $H\subseteq_{is} C_3$, $H\subseteq_{is} C_4$,\\
		&  and planar & $H \cong K_4^-$, or $H$ is edgeless  \\
		\hline
		\multirow{2}{*}{Minor} & \multirow{2}{*}{$H$ is planar} &
		$H\subseteq_{is} W_4$, $H \subseteq_{is} K_5^-$,\\
		&   &  $H \subseteq_{is} K_{2,q}$, or $H \subseteq_{is} K_{2,q}^+$ for some $q\in \mathbb{N}$\\
		\hline
	\end{tabular}
	\caption{Summary of $(\tw,\omega)$-bounded graph classes excluding a fixed graph $H$ for six graph containment relations.}\label{table-results}
\end{table}

To the best of our knowledge, these six dichotomies represent the first set of results towards a systematic study of the problem of classifying $(\tw,\omega)$-bounded graph classes.

One of the results from the table, namely, the  $(\tw,\omega)$-boundedness of the class of $K_{2,3}$\=/induced-minor-free graphs, implies that the class of $1$-perfectly orientable graphs is $(\tw,\omega)$-bounded. This answers a question of \citeauthor{MR3853110} posed in~\cite{MR3853110}.
Combining this result with an algorithmic result of Chaplick and Zeman from~\cite{DBLP:journals/endm/ChaplickZ17} shows that for any fixed $k$, there exists a linear-time algorithm for the $k$-coloring problem in the class of $1$-perfectly orientable graphs. This answers a question raised by \citeauthor{MR3992956} in~\cite{MR3992956}. 
Moreover, our results for the induced minor relation lead to an infinite family of $\chi$-bounded graph classes that were not previously known to be $\chi$-bounded: the classes of $H$-induced-minor-free graphs whenever $H$ is
isomorphic to $W_4$, $K_5^-$, $K_{2,q}$ for $q\ge 3$, or $K_{2,q}^+$ for $q\ge 3$.

From the algorithmic point of view, we observe that for any fixed positive integer $k$, the approach of Chaplick and Zeman from~\cite{DBLP:journals/endm/ChaplickZ17} can be adapted to obtain a robust polynomial-time algorithm for \textsc{List $k$-Coloring} in any graph class with a computable $(\tw,\omega)$-binding function. We also show how to approximate the clique number to within a factor of $\textsf{opt}^{1-1/\mathcal{O}(1)}$ in graph classes with a polynomially bounded $(\tw,\omega)$-binding function, where $\textsf{opt}$ is the clique number of the input graph.

Our techniques combine the development and applications of structural properties of graphs in restricted classes, connections with Hadwiger number and with minimal separators, as well as applications of Ramsey's theorem and known results on treewidth and graph minors. Results given by \cref{table-results} are derived in \cref{sec:isitm,sec:im,sec:stmm}. The algorithmic results are presented in \cref{sec:algo}.
In~\cref{sec:lower-bounds}, we show that there exists no polynomial that is a $(\tw,\omega)$-binding function for all polynomially  $(\tw,\omega)$-bounded graph classes.
In~\cref{section MWIS}, we consider the complexity of the \textsc{Maximum Weight Independent Set} problem in $(\tw,\omega)$-bounded graph classes; in this respect, we prove that the problem is polynomial-time solvable in every class of graphs excluding a fixed star as an induced minor.
We conclude the paper in~\cref{sec:discussion} with several open questions and research directions for further investigations of $(\tw,\omega)$-bounded graph classes.

\subsection{Related work}

The concept of a $(\tw,\omega)$-bounded graph class is part of the following more general framework. An (integer) \emph{graph invariant} is a mapping from the class of all graphs to the set of nonnegative integers $\N$ that does not distinguish between isomorphic graphs. 
Given two graph invariants $\rho$ and $\sigma$ and a graph class $\mathcal{G}$, we say that $\mathcal{G}$ is \emph{$(\rho,\sigma)$-bounded} if there exists a \emph{$(\rho,\sigma)$-binding function} for $\G$, that is, a  function $f:\N\to \N$ such that for every graph $G\in \G$ and every induced subgraph $G'$ of $G$, we have $\rho(G')\le f(\sigma(G'))$.
Probably the most well-known and well-studied case of $(\rho,\sigma)$-bounded graph classes corresponds to the pair $(\rho,\sigma) = (\chi,\omega)$, where $\chi(G)$ denotes the chromatic number of $G$. Such graph classes are called simply \emph{$\chi$-bounded}. They were introduced by Gy\'arf\'as in the late 1980s~to generalize perfection~\cite{MR951359} and studied extensively in the literature (see~\cite{MR4174126} for a survey).
Note that every graph $G$ satisfies
$\omega(G)\le \chi(G)\le \tw(G)+1$ (see~\cref{thm:chi}), where the first inequality holds with equality for all induced subgraphs of $G$ if and only if $G$ is perfect, and both inequalities hold with equality for all induced subgraphs of $G$ if and only if $G$ is chordal (see~\cref{thm:chordal}).
Thus, similarly as $\chi$-boundedness generalizes perfection, $(\tw,\omega)$-boundedness generalizes chordality, and every $(\tw,\omega)$-bounded graph class is also $\chi$-bounded (but not vice versa).

In their book on graph coloring problems~\cite{MR1304254}, Jensen and Toft referred to \hbox{$(\beta,\chi)$-bounded} graph families, where $\beta$ denotes the coloring number of $G$, as \textit{color-bound}.
Gy\'arf\'as and Zaker studied $(\delta, \chi)$-bounded graph classes~\cite{MR2811077}, where $\delta$ denotes the minimum degree of the graph. \citeauthor{MR3215457} showed in~\cite{MR3215457} that classes of intersection graphs of arithmetic progressions with bounded jumps are $(\textrm{pw},\omega)$-bounded, where $\textrm{pw}$ denotes the pathwidth of the graph. Several other variants of $(\rho,\sigma)$-bounded graph classes were studied in the literature, though not to the same extent as the $\chi$-bounded ones~(see, e.g.,~\cite{MR2795546,MR1385380,MR3794363}).

\citeauthor{MR3729840}~\cite{MR3729840} asked whether $(\tw,\omega)$-boundedness can be generalized from the class of chordal graphs to the class of even-hole-free graphs. While the answer is affirmative in the case of planar even-hole-free graphs~\cite{MR2652000}, the question was recently resolved in the negative by Sintiari and Trotignon~\cite{SintiariTrotignonJGT2021}.

Dichotomy studies similar to ours exist for many other properties of graph classes, including
$(\delta, \chi)$-boundedness~\cite{MR2811077},
boundedness of the clique-width~\cite{MR3967291,MR4031697}, well-quasi-ordering~\cite{MR3906632,MR3853108,MR1185012}, and polynomial-time solvability of \textsc{Graph Homomorphism}~\cite{MR1047555}, \textsc{Graph Isomorphism}~\cite{MR3712298}, \textsc{Dominating Set}~\cite{MR3515011}, and various coloring and packing problems~\cite{MR3623382,MR1234387,MR2401129}.

\section{Preliminaries}\label{sec:prelim}

We now define the six graph containment relations studied in this paper.
If a graph $H$ can be obtained from a graph $G$ by only deleting vertices, then $H$ is an \emph{induced subgraph} of $G$, and we write $H \subseteq_{is} G$.
If $H$ is obtained from $G$ by deleting vertices and edges, then $H$ is a \emph{subgraph} of $G$, and we write $H \subseteq_s G$.
Note that if $H \subseteq_{is} G$, then $H \subseteq_s G$. A subdivision of a graph $H$ is a graph obtained from $H$ by a sequence of edge subdivisions.
The subdivision of an edge $uv$ of a graph is the operation that removes the edge $uv$ and adds two edges $uw$ and $wv$, where $w$ is a new vertex.
The graph $H$ is said to be a \emph{topological minor}  (or \emph{topological subgraph}) of a graph $G$ if $G$ contains a subdivision of $H$ as a subgraph, and we write $H \subseteq_{tm} G$. Similarly, $H$ is an \emph{induced topological minor} of $G$ if $G$ contains a subdivision of $H$ as an induced subgraph, and we write $H \subseteq_{itm} G$. Again, if $H \subseteq_{itm} G$, then $H \subseteq_{tm} G$. An edge contraction is the operation of deleting a pair of adjacent vertices and replacing them with a new vertex whose neighborhood is the union of the neighborhoods of the two original vertices. We say that $G$ contains $H$ as \emph{induced minor} if $H$ can be obtained from $G$ by a sequence of vertex deletions and edge contractions, and we write $H \subseteq_{im} G$.
Finally, if $H$ can be obtained from $G$ by a sequence of vertex deletions, edge deletions, and edge contractions, then $H$ is said to be a \emph{minor} of $G$, and we write $H \subseteq_{m} G$.
Here also, if $H \subseteq_{im} G$, then $H \subseteq_{m} G$.
Besides the already observed implications, one can notice that
\[
\begin{array}{cccccc}
    H \subseteq_{s} G   & \implies & H \subseteq_{tm} G     & \implies & H \subseteq_{m} G\phantom{\,.} &\text{and}\\
    H \subseteq_{is} G  & \implies & H \subseteq_{itm} G    & \implies & H \subseteq_{im} G\,.&
\end{array}
\]

If $G$ does not contain an induced subgraph isomorphic to $H$, then we say that $G$ is $H$-free. Analogously, we may also say that $G$ is $H$-subgraph-free, $H$-topological-minor-free, $H$-induced-topological-minor-free, $H$-minor-free, or $H$-induced-minor-free, respectively, for the other five relations.
This terminology and notation is naturally extended to the case of finitely many forbidden graphs with respect to any of the six graph containment relations.
For example, a graph $G$ is said to be $\{H_1,\ldots, H_p\}$-free if $G$ is $H_i$-free for all $i\in \{1,\ldots, p\}$.

It is well known that $G$ contains $H$ as a minor if and only if there exists a \emph{minor model} of $H$ in $G$, that is, a collection $(X_u: u\in V(H))$ of pairwise disjoint subsets of $V(G)$ called \emph{bags} such that each $X_u$ induces a connected subgraph of $G$ and for every two adjacent vertices $u,v\in V(H)$, there is an edge in $G$ between a vertex of $X_u$ and a vertex of $X_v$. Similarly, $G$ contains $H$ as an induced minor if and only if there exists an \emph{induced minor model} of $H$ in $G$, which is defined similarly as a minor model, except that for every two distinct vertices $u,v\in V(H)$, there is an edge in $G$ between a vertex of $X_u$ and a vertex of $X_v$ if and only if $uv \in E(H)$.

Given a set $S \subseteq V(G)$, we denote by $G-S$ the graph obtained from $G$ by removing all vertices in $S$ and by $G[S]$ the \emph{subgraph of $G$ induced by $S$}, that is, the graph $G-(V(G)\setminus S)$.
For $u \in V$, $N(u) = \{v \in V : uv \in E\}$ is the \emph{neighborhood} of $u$ and $N[u] = N(u) \cup \{u\}$ is the \emph{closed neighborhood} of $u$. The \emph{degree} of $u$ in $G$ is denoted by $d_G(u)$ and defined as the cardinality of its neighborhood.
A \emph{clique} in a graph $G$ is a set of pairwise adjacent vertices, and an \emph{independent set} is a set of pairwise nonadjacent vertices. The \emph{clique number} of a graph $G$, denoted by $\omega(G)$, is the maximum size of a clique in $G$.
The \emph{independence number} of a graph $G$, denoted by $\alpha(G)$, is the maximum size of an independent set in $G$.

A \emph{tree decomposition} of a graph $G$ is a pair $(T, \{X_t : t \in V(T)\})$, where $T$ is a tree and each $t \in V(T)$ is associated with a vertex subset $X_t \subseteq V(G)$ such that $\bigcup_{t \in V(T)} X_t = V$,
for each edge $uv \in E(G)$ there exists some $t\in V(T)$ such that $u,v \in X_t$, and for every $u \in V(G)$, the set $T_u = \{t \in V(T) : u \in X_t\}$ induces a connected subtree of $T$.
The \emph{width} of a tree decomposition equals $\max_{t \in V(T)} |X_t|-1$, and the \emph{treewidth} of a graph $G$, denoted by $\tw(G)$, is the minimum possible width of a tree decomposition of $G$. A graph class $\G$ is said to be \emph{of bounded treewidth} (or to \emph{have bounded treewidth}) if there exists a constant $c$ such that $\tw(G)\le c$ for all $G\in \G$; otherwise, $\G$ is of \emph{unbounded treewidth} (or \emph{has unbounded treewidth}). A \emph{hole} in a graph $G$ is an induced subgraph of $G$ isomorphic to a cycle of length at least four. A graph is said to be \emph{chordal} if it does not contain any hole.

Treewidth can be defined in many equivalent ways. One of the characterizations states that the treewidth of a graph $G$ equals
the minimum value of $\omega(G')-1$ such that $G$ is a subgraph of $G'$ and $G'$ is chordal (see, e.g.,~\cite{MR1647486}).
In particular, this characterization implies the following.

\begin{theorem}\label{thm:chordal}
Every graph $G$ satisfies $\tw(G) \ge \omega(G)-1$ with equality for all induced subgraphs if and only if $G$ is chordal.
\end{theorem}

Since chordal graphs are perfect, their clique and chromatic numbers coincide, so an equivalent characterization is that the treewidth of a graph $G$ equals the minimum value of $\chi(G')-1$ such that $G$ is a subgraph of $G'$ and $G'$ is chordal. Fixing such a chordal graph $G'$ with $\chi(G') -1 = \tw(G)$ and using the fact that
chromatic number is monotone under subgraphs, the following
strengthening of the inequality given by \cref{thm:chordal} holds.

\begin{theorem}\label{thm:chi}
Every graph $G$ satisfies $\tw(G) \ge \chi(G)-1$.
\end{theorem}

Let $\G$ be a $(\tw,\omega)$-bounded graph class with a binding function $f$.
From \cref{thm:chi}, we obtain that $\chi(G)-1 \leq \tw(G) \leq f(\omega(G))$ whenever $G$ is an induced subgraph of a graph in $\G$.
Hence, we obtain the following corollary.

\begin{corollary}\label{corollary:(tw omega)-B implies chi-B}
Every $(\tw,\omega)$-bounded graph class is $\chi$-bounded.
\end{corollary}

The following observation is an immediate consequence of the definitions.

\begin{observation}\label{observation:chordal}
Let $G$ be a graph. Then, the following conditions are equivalent:
\begin{enumerate}
    \item $G$ is chordal.
    \item $G$ is $C_4$-induced-minor-free.
    \item $G$ is $C_4$-induced-topological-minor-free.
\end{enumerate}
\end{observation}

Some of our proofs will make use of the following classical result due to Ramsey~\cite{MR1576401}.

\begin{theorem}[Ramsey's theorem]
For every two positive integers $k$ and $\ell$, there exists a least positive integer $R(k,\ell)$ such that every graph with at least $R(k,\ell)$ vertices contains either a clique of size $k$ or an independent set of size $\ell$.
\end{theorem}

The standard proof of Ramsey's theorem is based on the inequality
$R(k,\ell) \le  R(k-1, \ell) + R(k, \ell-1)$ for all $k,\ell\ge 2$, which implies that $R(k,\ell)\le \binom{k+\ell-2}{k-1}$ for all positive integers $k$ and $\ell$.

Using Ramsey's theorem, we can already derive the following.

\begin{lemma}\label{H is edgeless}
Let $H$ be an edgeless graph. Then the class of $H$-free graphs is $(\tw,\omega)$-bounded with a binding function  $f(k) = R(k+1,|V(H)|)-2$, which is bounded by a polynomial in $k$ of degree $|V(H)|-1$.
\end{lemma}
\begin{proof}
Let $k \in \mathbb{N}$ and let $G$ be an $H$-free graph such that $\omega(G) = k$.
Since $H$ is edgeless, Ramsey's theorem implies that the number of vertices in $G$ is strictly smaller than $R(k+1,|V(H)|)$. In particular, the treewidth of $G$ is at most $|V(G)|-1\le R(k+1,|V(H)|)-2$.
\end{proof}

A graph class that is not $(\tw,\omega)$-bounded is said to be \emph{$(\tw,\omega)$-unbounded}.
Some specific $(\tw,\omega)$-unbounded graph classes, which will play a crucial role in our proofs, are discussed in \cref{lemma:H is not any H}.
The \emph{line graph} of a graph $G$, denoted by $L(G)$, is the graph with vertex set $E(G)$ where two vertices are adjacent if and only if the corresponding edges intersect.
For the definition of an \emph{elementary wall}, we refer to~\cite{MR3451044}.
For a nonnegative integer $q$, we say that a graph is a \emph{$q$-subdivided-wall} if it can be obtained from an elementary wall by subdividing each edge $q$ times.
See \cref{fig:walls and their line graphs}
for an illustration of an elementary wall, a $1$-subdivided wall, and the line graph of a $1$-subdivided wall.

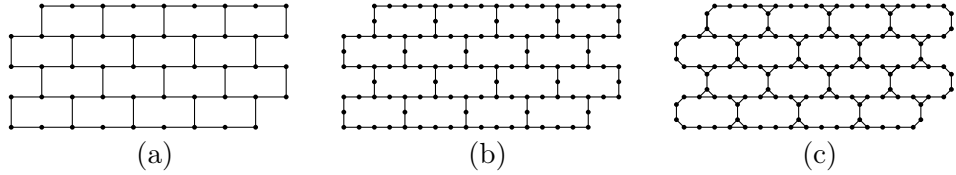
\begin{figure}[htbp]
\centering
\tikzstyle{ipe stylesheet} = [
  ipe import,
  even odd rule,
  line join=round,
  line cap=butt,
  ipe pen normal/.style={line width=0.4},
  ipe pen heavier/.style={line width=0.8},
  ipe pen fat/.style={line width=1.2},
  ipe pen ultrafat/.style={line width=2},
  ipe pen normal,
  ipe mark normal/.style={ipe mark scale=3},
  ipe mark large/.style={ipe mark scale=5},
  ipe mark small/.style={ipe mark scale=1.5},
  ipe mark tiny/.style={ipe mark scale=1.1},
  ipe mark normal,
  /pgf/arrow keys/.cd,
  ipe arrow normal/.style={scale=7},
  ipe arrow large/.style={scale=10},
  ipe arrow small/.style={scale=5},
  ipe arrow tiny/.style={scale=3},
  ipe arrow normal,
  /tikz/.cd,
  ipe arrows,
  <->/.tip = ipe normal,
  ipe dash normal/.style={dash pattern=},
  ipe dash dashed/.style={dash pattern=on 4bp off 4bp},
  ipe dash dotted/.style={dash pattern=on 1bp off 3bp},
  ipe dash dash dotted/.style={dash pattern=on 4bp off 2bp on 1bp off 2bp},
  ipe dash dash dot dotted/.style={dash pattern=on 4bp off 2bp on 1bp off 2bp on 1bp off 2bp},
  ipe dash normal,
  ipe node/.append style={font=\normalsize},
  ipe stretch normal/.style={ipe node stretch=1},
  ipe stretch normal,
  ipe opacity opaque/.style={opacity=1},
  ipe opacity opaque,
]
\begin{tikzpicture}[ipe stylesheet,scale=0.62]
  \node[ipe node]
     at (91.193, 682.14) {(a)};
  \draw[shift={(16, 704)}, xscale=0.7655, yscale=0.5741]
    (0, 0) rectangle (48, 32);
  \draw[shift={(52.743, 704)}, xscale=0.7655, yscale=0.5741]
    (0, 0)
     -- (48, 0)
     -- (48, 32)
     -- (0, 32);
  \draw[shift={(89.487, 704)}, xscale=0.7655, yscale=0.5741]
    (0, 0)
     -- (48, 0)
     -- (48, 32)
     -- (0, 32);
  \draw[shift={(126.23, 704)}, xscale=0.7655, yscale=0.5741]
    (0, 0)
     -- (48, 0)
     -- (48, 32)
     -- (0, 32);
  \draw[shift={(34.372, 722.371)}, xscale=0.7655, yscale=0.5741]
    (0, 0)
     -- (0, 32)
     -- (48, 32)
     -- (48, 0);
  \draw[shift={(71.115, 740.743)}, xscale=0.7655, yscale=0.5741]
    (0, 0)
     -- (48, 0)
     -- (48, -32);
  \draw[shift={(107.858, 740.743)}, xscale=0.7655, yscale=0.5741]
    (0, 0)
     -- (48, 0)
     -- (48, -32);
  \draw[shift={(144.601, 740.743)}, xscale=0.7655, yscale=0.5741]
    (0, 0)
     -- (48, 0)
     -- (48, -32)
     -- (24, -32);
  \draw[shift={(34.372, 759.115)}, xscale=0.7655, yscale=0.5741]
    (0, 0)
     -- (0, 32)
     -- (48, 32)
     -- (48, 0);
  \draw[shift={(71.115, 777.487)}, xscale=0.7655, yscale=0.5741]
    (0, 0)
     -- (48, 0)
     -- (48, -32);
  \draw[shift={(107.858, 777.487)}, xscale=0.7655, yscale=0.5741]
    (0, 0)
     -- (48, 0)
     -- (48, -32);
  \draw[shift={(144.601, 777.487)}, xscale=0.7655, yscale=0.5741]
    (0, 0)
     -- (48, 0)
     -- (48, -32)
     -- (24, -32);
  \draw[shift={(34.372, 740.743)}, xscale=0.7655, yscale=0.5741]
    (0, 0)
     -- (-24, 0)
     -- (-24, 32)
     -- (168, 32)
     -- (168, 0);
  \draw[shift={(52.743, 740.743)}, xscale=0.7655, yscale=0.5741]
    (0, 0)
     -- (0, 32);
  \draw[shift={(89.487, 740.743)}, xscale=0.7655, yscale=0.5741]
    (0, 0)
     -- (0, 32);
  \draw[shift={(126.23, 740.743)}, xscale=0.7655, yscale=0.5741]
    (0, 0)
     -- (0, 32);
  \pic[ipe mark scale=1.5]
     at (16, 759.1147) {ipe disk};
  \pic[ipe mark scale=1.5]
     at (16, 740.7431) {ipe disk};
  \pic[ipe mark scale=1.5]
     at (34.3716, 740.7431) {ipe disk};
  \pic[ipe mark scale=1.5]
     at (34.3716, 722.3714) {ipe disk};
  \pic[ipe mark scale=1.5]
     at (16, 722.3714) {ipe disk};
  \pic[ipe mark scale=1.5]
     at (16, 703.9998) {ipe disk};
  \pic[ipe mark scale=1.5]
     at (52.7433, 703.9998) {ipe disk};
  \pic[ipe mark scale=1.5]
     at (52.7433, 722.3714) {ipe disk};
  \pic[ipe mark scale=1.5]
     at (71.1149, 722.3714) {ipe disk};
  \pic[ipe mark scale=1.5]
     at (71.1149, 740.7431) {ipe disk};
  \pic[ipe mark scale=1.5]
     at (34.3716, 703.9998) {ipe disk};
  \pic[ipe mark scale=1.5]
     at (71.1149, 703.9998) {ipe disk};
  \pic[ipe mark scale=1.5]
     at (89.4866, 703.9998) {ipe disk};
  \pic[ipe mark scale=1.5]
     at (89.4866, 722.3714) {ipe disk};
  \pic[ipe mark scale=1.5]
     at (107.8582, 722.3714) {ipe disk};
  \pic[ipe mark scale=1.5]
     at (126.2298, 722.3714) {ipe disk};
  \pic[ipe mark scale=1.5]
     at (126.2298, 703.9998) {ipe disk};
  \pic[ipe mark scale=1.5]
     at (107.8582, 703.9998) {ipe disk};
  \pic[ipe mark scale=1.5]
     at (144.6015, 703.9998) {ipe disk};
  \pic[ipe mark scale=1.5]
     at (144.6015, 722.3714) {ipe disk};
  \pic[ipe mark scale=1.5]
     at (162.9731, 722.3714) {ipe disk};
  \pic[ipe mark scale=1.5]
     at (162.9731, 703.9998) {ipe disk};
  \pic[ipe mark scale=1.5]
     at (181.3448, 722.3714) {ipe disk};
  \pic[ipe mark scale=1.5]
     at (52.7433, 740.7431) {ipe disk};
  \pic[ipe mark scale=1.5]
     at (52.7433, 759.1147) {ipe disk};
  \pic[ipe mark scale=1.5]
     at (34.3716, 759.1147) {ipe disk};
  \pic[ipe mark scale=1.5]
     at (34.3716, 777.4864) {ipe disk};
  \pic[ipe mark scale=1.5]
     at (52.7433, 777.4864) {ipe disk};
  \pic[ipe mark scale=1.5]
     at (71.1149, 777.4864) {ipe disk};
  \pic[ipe mark scale=1.5]
     at (71.1149, 759.1147) {ipe disk};
  \pic[ipe mark scale=1.5]
     at (89.4866, 759.1147) {ipe disk};
  \pic[ipe mark scale=1.5]
     at (89.4866, 777.4864) {ipe disk};
  \pic[ipe mark scale=1.5]
     at (107.8582, 777.4864) {ipe disk};
  \pic[ipe mark scale=1.5]
     at (107.8582, 759.1147) {ipe disk};
  \pic[ipe mark scale=1.5]
     at (126.2298, 759.1147) {ipe disk};
  \pic[ipe mark scale=1.5]
     at (107.8582, 740.7431) {ipe disk};
  \pic[ipe mark scale=1.5]
     at (89.4866, 740.7431) {ipe disk};
  \pic[ipe mark scale=1.5]
     at (126.2298, 740.7431) {ipe disk};
  \pic[ipe mark scale=1.5]
     at (144.6015, 759.1147) {ipe disk};
  \pic[ipe mark scale=1.5]
     at (144.6015, 777.4864) {ipe disk};
  \pic[ipe mark scale=1.5]
     at (162.9731, 759.1147) {ipe disk};
  \pic[ipe mark scale=1.5]
     at (181.3448, 777.4864) {ipe disk};
  \pic[ipe mark scale=1.5]
     at (181.3448, 759.1147) {ipe disk};
  \pic[ipe mark scale=1.5]
     at (162.9731, 740.7431) {ipe disk};
  \pic[ipe mark scale=1.5]
     at (144.6015, 740.7431) {ipe disk};
  \pic[ipe mark scale=1.5]
     at (181.3448, 740.7431) {ipe disk};
  \draw[shift={(216, 704)}, xscale=0.7655, yscale=0.5741]
    (0, 0) rectangle (48, 32);
  \draw[shift={(252.743, 704)}, xscale=0.7655, yscale=0.5741]
    (0, 0)
     -- (48, 0)
     -- (48, 32)
     -- (0, 32);
  \draw[shift={(289.487, 704)}, xscale=0.7655, yscale=0.5741]
    (0, 0)
     -- (48, 0)
     -- (48, 32)
     -- (0, 32);
  \draw[shift={(326.23, 704)}, xscale=0.7655, yscale=0.5741]
    (0, 0)
     -- (48, 0)
     -- (48, 32)
     -- (0, 32);
  \draw[shift={(234.371, 722.371)}, xscale=0.7655, yscale=0.5741]
    (0, 0)
     -- (0, 32)
     -- (48, 32)
     -- (48, 0);
  \draw[shift={(271.115, 740.743)}, xscale=0.7655, yscale=0.5741]
    (0, 0)
     -- (48, 0)
     -- (48, -32);
  \draw[shift={(307.859, 740.743)}, xscale=0.7655, yscale=0.5741]
    (0, 0)
     -- (48, 0)
     -- (48, -32);
  \draw[shift={(344.601, 740.743)}, xscale=0.7655, yscale=0.5741]
    (0, 0)
     -- (48, 0)
     -- (48, -32)
     -- (24, -32);
  \draw[shift={(234.371, 759.115)}, xscale=0.7655, yscale=0.5741]
    (0, 0)
     -- (0, 32)
     -- (48, 32)
     -- (48, 0);
  \draw[shift={(271.115, 777.487)}, xscale=0.7655, yscale=0.5741]
    (0, 0)
     -- (48, 0)
     -- (48, -32);
  \draw[shift={(307.859, 777.487)}, xscale=0.7655, yscale=0.5741]
    (0, 0)
     -- (48, 0)
     -- (48, -32);
  \draw[shift={(344.601, 777.487)}, xscale=0.7655, yscale=0.5741]
    (0, 0)
     -- (48, 0)
     -- (48, -32)
     -- (24, -32);
  \draw[shift={(234.371, 740.743)}, xscale=0.7655, yscale=0.5741]
    (0, 0)
     -- (-24, 0)
     -- (-24, 32)
     -- (168, 32)
     -- (168, 0);
  \draw[shift={(252.743, 740.743)}, xscale=0.7655, yscale=0.5741]
    (0, 0)
     -- (0, 32);
  \draw[shift={(289.486, 740.743)}, xscale=0.7655, yscale=0.5741]
    (0, 0)
     -- (0, 32);
  \draw[shift={(326.23, 740.743)}, xscale=0.7655, yscale=0.5741]
    (0, 0)
     -- (0, 32);
  \pic[ipe mark scale=1.5]
     at (215.9996, 759.1147) {ipe disk};
  \pic[ipe mark scale=1.5]
     at (215.9996, 740.7431) {ipe disk};
  \pic[ipe mark scale=1.5]
     at (234.3713, 740.7431) {ipe disk};
  \pic[ipe mark scale=1.5]
     at (234.3713, 722.3714) {ipe disk};
  \pic[ipe mark scale=1.5]
     at (215.9996, 722.3714) {ipe disk};
  \pic[ipe mark scale=1.5]
     at (252.743, 703.9998) {ipe disk};
  \pic[ipe mark scale=1.5]
     at (252.743, 722.3714) {ipe disk};
  \pic[ipe mark scale=1.5]
     at (271.1146, 722.3714) {ipe disk};
  \pic[ipe mark scale=1.5]
     at (271.1146, 740.7431) {ipe disk};
  \pic[ipe mark scale=1.5]
     at (234.3713, 703.9998) {ipe disk};
  \pic[ipe mark scale=1.5]
     at (271.1146, 703.9998) {ipe disk};
  \pic[ipe mark scale=1.5]
     at (289.4863, 703.9998) {ipe disk};
  \pic[ipe mark scale=1.5]
     at (289.4863, 722.3714) {ipe disk};
  \pic[ipe mark scale=1.5]
     at (307.858, 722.3714) {ipe disk};
  \pic[ipe mark scale=1.5]
     at (326.2297, 722.3714) {ipe disk};
  \pic[ipe mark scale=1.5]
     at (326.2297, 703.9998) {ipe disk};
  \pic[ipe mark scale=1.5]
     at (307.858, 703.9998) {ipe disk};
  \pic[ipe mark scale=1.5]
     at (344.6014, 703.9998) {ipe disk};
  \pic[ipe mark scale=1.5]
     at (344.6014, 722.3714) {ipe disk};
  \pic[ipe mark scale=1.5]
     at (362.9731, 722.3714) {ipe disk};
  \pic[ipe mark scale=1.5]
     at (362.9731, 703.9998) {ipe disk};
  \pic[ipe mark scale=1.5]
     at (381.3448, 722.3714) {ipe disk};
  \pic[ipe mark scale=1.5]
     at (252.743, 740.7431) {ipe disk};
  \pic[ipe mark scale=1.5]
     at (252.743, 759.1147) {ipe disk};
  \pic[ipe mark scale=1.5]
     at (234.3713, 759.1147) {ipe disk};
  \pic[ipe mark scale=1.5]
     at (234.3713, 777.4864) {ipe disk};
  \pic[ipe mark scale=1.5]
     at (252.743, 777.4864) {ipe disk};
  \pic[ipe mark scale=1.5]
     at (271.1146, 777.4864) {ipe disk};
  \pic[ipe mark scale=1.5]
     at (271.1146, 759.1147) {ipe disk};
  \pic[ipe mark scale=1.5]
     at (289.4863, 759.1147) {ipe disk};
  \pic[ipe mark scale=1.5]
     at (289.4863, 777.4864) {ipe disk};
  \pic[ipe mark scale=1.5]
     at (307.858, 777.4864) {ipe disk};
  \pic[ipe mark scale=1.5]
     at (307.858, 759.1147) {ipe disk};
  \pic[ipe mark scale=1.5]
     at (326.2297, 759.1147) {ipe disk};
  \pic[ipe mark scale=1.5]
     at (307.858, 740.7431) {ipe disk};
  \pic[ipe mark scale=1.5]
     at (289.4863, 740.7431) {ipe disk};
  \pic[ipe mark scale=1.5]
     at (289.4863, 749.9289) {ipe disk};
  \pic[ipe mark scale=1.5]
     at (326.2297, 740.7431) {ipe disk};
  \pic[ipe mark scale=1.5]
     at (344.6014, 759.1147) {ipe disk};
  \pic[ipe mark scale=1.5]
     at (344.6014, 777.4864) {ipe disk};
  \pic[ipe mark scale=1.5]
     at (362.9731, 759.1147) {ipe disk};
  \pic[ipe mark scale=1.5]
     at (381.3448, 759.1147) {ipe disk};
  \pic[ipe mark scale=1.5]
     at (362.9731, 740.7431) {ipe disk};
  \pic[ipe mark scale=1.5]
     at (344.6014, 740.7431) {ipe disk};
  \pic[ipe mark scale=1.5]
     at (381.3448, 740.7431) {ipe disk};
  \pic[ipe mark scale=1.5]
     at (225.1857, 759.1147) {ipe disk};
  \pic[ipe mark scale=1.5]
     at (215.9999, 749.9289) {ipe disk};
  \pic[ipe mark scale=1.5]
     at (225.1857, 740.7431) {ipe disk};
  \pic[ipe mark scale=1.5]
     at (243.5574, 740.7431) {ipe disk};
  \pic[ipe mark scale=1.5]
     at (243.5574, 759.1147) {ipe disk};
  \pic[ipe mark scale=1.5]
     at (234.3716, 768.3006) {ipe disk};
  \pic[ipe mark scale=1.5]
     at (243.5574, 777.4864) {ipe disk};
  \pic[ipe mark scale=1.5]
     at (261.9291, 777.4864) {ipe disk};
  \pic[ipe mark scale=1.5]
     at (261.9291, 759.1147) {ipe disk};
  \pic[ipe mark scale=1.5]
     at (271.115, 768.3006) {ipe disk};
  \pic[ipe mark scale=1.5]
     at (280.3008, 777.4864) {ipe disk};
  \pic[ipe mark scale=1.5]
     at (298.6726, 777.4864) {ipe disk};
  \pic[ipe mark scale=1.5]
     at (317.0443, 777.4864) {ipe disk};
  \pic[ipe mark scale=1.5]
     at (326.2301, 777.4864) {ipe disk};
  \pic[ipe mark scale=1.5]
     at (335.416, 777.4864) {ipe disk};
  \pic[ipe mark scale=1.5]
     at (353.7877, 777.4864) {ipe disk};
  \pic[ipe mark scale=1.5]
     at (362.9736, 777.4864) {ipe disk};
  \pic[ipe mark scale=1.5]
     at (372.1594, 777.4864) {ipe disk};
  \pic[ipe mark scale=1.5]
     at (381.3453, 768.3006) {ipe disk};
  \pic[ipe mark scale=1.5]
     at (344.6018, 768.3006) {ipe disk};
  \pic[ipe mark scale=1.5]
     at (307.8584, 768.3006) {ipe disk};
  \pic[ipe mark scale=1.5]
     at (280.3008, 759.1147) {ipe disk};
  \pic[ipe mark scale=1.5]
     at (298.6726, 759.1147) {ipe disk};
  \pic[ipe mark scale=1.5]
     at (317.0443, 759.1147) {ipe disk};
  \pic[ipe mark scale=1.5]
     at (335.416, 759.1147) {ipe disk};
  \pic[ipe mark scale=1.5]
     at (353.7877, 759.1147) {ipe disk};
  \pic[ipe mark scale=1.5]
     at (362.9736, 749.9289) {ipe disk};
  \pic[ipe mark scale=1.5]
     at (353.7877, 740.7431) {ipe disk};
  \pic[ipe mark scale=1.5]
     at (335.416, 740.7431) {ipe disk};
  \pic[ipe mark scale=1.5]
     at (326.2301, 749.9289) {ipe disk};
  \pic[ipe mark scale=1.5]
     at (298.6726, 740.7431) {ipe disk};
  \pic[ipe mark scale=1.5]
     at (280.3008, 740.7431) {ipe disk};
  \pic[ipe mark scale=1.5]
     at (261.9291, 740.7431) {ipe disk};
  \pic[ipe mark scale=1.5]
     at (252.7433, 749.9289) {ipe disk};
  \pic[ipe mark scale=1.5]
     at (234.3716, 731.5573) {ipe disk};
  \pic[ipe mark scale=1.5]
     at (225.1857, 722.3714) {ipe disk};
  \pic[ipe mark scale=1.5]
     at (215.9999, 713.1856) {ipe disk};
  \pic[ipe mark scale=1.5]
     at (225.1857, 703.9998) {ipe disk};
  \pic[ipe mark scale=1.5]
     at (243.5574, 703.9998) {ipe disk};
  \pic[ipe mark scale=1.5]
     at (252.7433, 713.1856) {ipe disk};
  \pic[ipe mark scale=1.5]
     at (243.5574, 722.3714) {ipe disk};
  \pic[ipe mark scale=1.5]
     at (261.9291, 722.3714) {ipe disk};
  \pic[ipe mark scale=1.5]
     at (271.115, 731.5573) {ipe disk};
  \pic[ipe mark scale=1.5]
     at (280.3008, 722.3714) {ipe disk};
  \pic[ipe mark scale=1.5]
     at (289.4867, 713.1856) {ipe disk};
  \pic[ipe mark scale=1.5]
     at (280.3008, 703.9998) {ipe disk};
  \pic[ipe mark scale=1.5]
     at (261.9291, 703.9998) {ipe disk};
  \pic[ipe mark scale=1.5]
     at (298.6726, 722.3714) {ipe disk};
  \pic[ipe mark scale=1.5]
     at (298.6726, 703.9998) {ipe disk};
  \pic[ipe mark scale=1.5]
     at (317.0443, 703.9998) {ipe disk};
  \pic[ipe mark scale=1.5]
     at (317.0443, 722.3714) {ipe disk};
  \pic[ipe mark scale=1.5]
     at (326.2301, 713.1856) {ipe disk};
  \pic[ipe mark scale=1.5]
     at (335.416, 703.9998) {ipe disk};
  \pic[ipe mark scale=1.5]
     at (353.7877, 703.9998) {ipe disk};
  \pic[ipe mark scale=1.5]
     at (362.9736, 713.1856) {ipe disk};
  \pic[ipe mark scale=1.5]
     at (353.7877, 722.3714) {ipe disk};
  \pic[ipe mark scale=1.5]
     at (335.416, 722.3714) {ipe disk};
  \pic[ipe mark scale=1.5]
     at (344.6018, 731.5573) {ipe disk};
  \pic[ipe mark scale=1.5]
     at (307.8584, 731.5573) {ipe disk};
  \pic[ipe mark scale=1.5]
     at (372.1594, 722.3714) {ipe disk};
  \pic[ipe mark scale=1.5]
     at (381.3453, 731.5573) {ipe disk};
  \pic[ipe mark scale=1.5]
     at (372.1594, 740.7431) {ipe disk};
  \pic[ipe mark scale=1.5]
     at (372.1594, 759.1147) {ipe disk};
  \pic[ipe mark scale=1.5]
     at (317.0443, 740.7431) {ipe disk};
  \pic[ipe mark scale=1.5]
     at (126.2296, 777.4864) {ipe disk};
  \pic[ipe mark scale=1.5]
     at (162.9726, 777.4864) {ipe disk};
  \pic[ipe mark scale=1.5]
     at (438.965, 777.4868) {ipe disk};
  \pic[ipe mark scale=1.5]
     at (448.1509, 777.4868) {ipe disk};
  \pic[ipe mark scale=1.5]
     at (457.3367, 777.4868) {ipe disk};
  \pic[ipe mark scale=1.5]
     at (466.5225, 777.4868) {ipe disk};
  \pic[ipe mark scale=1.5]
     at (471.1154, 772.8939) {ipe disk};
  \pic[ipe mark scale=1.5]
     at (475.7083, 777.4868) {ipe disk};
  \pic[ipe mark scale=1.5]
     at (484.8942, 777.4868) {ipe disk};
  \pic[ipe mark scale=1.5]
     at (494.08, 777.4868) {ipe disk};
  \pic[ipe mark scale=1.5]
     at (503.2658, 777.4868) {ipe disk};
  \pic[ipe mark scale=1.5]
     at (507.8587, 772.8939) {ipe disk};
  \pic[ipe mark scale=1.5]
     at (512.4516, 777.4868) {ipe disk};
  \pic[ipe mark scale=1.5]
     at (521.6374, 777.4868) {ipe disk};
  \pic[ipe mark scale=1.5]
     at (530.8233, 777.4868) {ipe disk};
  \pic[ipe mark scale=1.5]
     at (540.0091, 777.4868) {ipe disk};
  \pic[ipe mark scale=1.5]
     at (544.602, 772.8939) {ipe disk};
  \pic[ipe mark scale=1.5]
     at (549.1949, 777.4868) {ipe disk};
  \pic[ipe mark scale=1.5]
     at (558.3807, 777.4868) {ipe disk};
  \pic[ipe mark scale=1.5]
     at (567.5666, 777.4868) {ipe disk};
  \pic[ipe mark scale=1.5]
     at (576.7524, 777.4868) {ipe disk};
  \pic[ipe mark scale=1.5]
     at (581.3453, 772.8939) {ipe disk};
  \pic[ipe mark scale=1.5]
     at (581.3453, 763.708) {ipe disk};
  \pic[ipe mark scale=1.5]
     at (576.7524, 759.1151) {ipe disk};
  \pic[ipe mark scale=1.5]
     at (567.5666, 759.1151) {ipe disk};
  \pic[ipe mark scale=1.5]
     at (558.3807, 759.1151) {ipe disk};
  \pic[ipe mark scale=1.5]
     at (549.1949, 759.1151) {ipe disk};
  \pic[ipe mark scale=1.5]
     at (544.602, 763.708) {ipe disk};
  \pic[ipe mark scale=1.5]
     at (540.0091, 759.1151) {ipe disk};
  \pic[ipe mark scale=1.5]
     at (530.8233, 759.1151) {ipe disk};
  \pic[ipe mark scale=1.5]
     at (521.6374, 759.1151) {ipe disk};
  \pic[ipe mark scale=1.5]
     at (512.4516, 759.1151) {ipe disk};
  \pic[ipe mark scale=1.5]
     at (507.8587, 763.708) {ipe disk};
  \pic[ipe mark scale=1.5]
     at (503.2658, 759.1151) {ipe disk};
  \pic[ipe mark scale=1.5]
     at (494.08, 759.1151) {ipe disk};
  \pic[ipe mark scale=1.5]
     at (484.8942, 759.1151) {ipe disk};
  \pic[ipe mark scale=1.5]
     at (475.7083, 759.1151) {ipe disk};
  \pic[ipe mark scale=1.5]
     at (471.1154, 763.708) {ipe disk};
  \pic[ipe mark scale=1.5]
     at (466.5225, 759.1151) {ipe disk};
  \pic[ipe mark scale=1.5]
     at (457.3367, 759.1151) {ipe disk};
  \pic[ipe mark scale=1.5]
     at (448.1509, 759.1151) {ipe disk};
  \pic[ipe mark scale=1.5]
     at (438.965, 759.1151) {ipe disk};
  \pic[ipe mark scale=1.5]
     at (429.7792, 759.1151) {ipe disk};
  \pic[ipe mark scale=1.5]
     at (420.5934, 759.1151) {ipe disk};
  \pic[ipe mark scale=1.5]
     at (416.0005, 754.5222) {ipe disk};
  \pic[ipe mark scale=1.5]
     at (416.0005, 745.3364) {ipe disk};
  \pic[ipe mark scale=1.5]
     at (420.5934, 740.7435) {ipe disk};
  \pic[ipe mark scale=1.5]
     at (429.7792, 740.7435) {ipe disk};
  \pic[ipe mark scale=1.5]
     at (438.965, 740.7435) {ipe disk};
  \pic[ipe mark scale=1.5]
     at (448.1509, 740.7435) {ipe disk};
  \pic[ipe mark scale=1.5]
     at (452.7438, 745.3364) {ipe disk};
  \pic[ipe mark scale=1.5]
     at (452.7438, 754.5222) {ipe disk};
  \pic[ipe mark scale=1.5]
     at (457.3367, 740.7435) {ipe disk};
  \pic[ipe mark scale=1.5]
     at (466.5225, 740.7435) {ipe disk};
  \pic[ipe mark scale=1.5]
     at (475.7083, 740.7435) {ipe disk};
  \pic[ipe mark scale=1.5]
     at (484.8942, 740.7435) {ipe disk};
  \pic[ipe mark scale=1.5]
     at (489.4871, 745.3364) {ipe disk};
  \pic[ipe mark scale=1.5]
     at (489.4871, 754.5222) {ipe disk};
  \pic[ipe mark scale=1.5]
     at (494.08, 740.7435) {ipe disk};
  \pic[ipe mark scale=1.5]
     at (503.2658, 740.7435) {ipe disk};
  \pic[ipe mark scale=1.5]
     at (512.4516, 740.7435) {ipe disk};
  \pic[ipe mark scale=1.5]
     at (521.6374, 740.7435) {ipe disk};
  \pic[ipe mark scale=1.5]
     at (526.2304, 745.3364) {ipe disk};
  \pic[ipe mark scale=1.5]
     at (526.2304, 754.5222) {ipe disk};
  \pic[ipe mark scale=1.5]
     at (530.8233, 740.7435) {ipe disk};
  \pic[ipe mark scale=1.5]
     at (540.0091, 740.7435) {ipe disk};
  \pic[ipe mark scale=1.5]
     at (549.1949, 740.7435) {ipe disk};
  \pic[ipe mark scale=1.5]
     at (558.3807, 740.7435) {ipe disk};
  \pic[ipe mark scale=1.5]
     at (562.9737, 745.3364) {ipe disk};
  \pic[ipe mark scale=1.5]
     at (562.9737, 754.5222) {ipe disk};
  \pic[ipe mark scale=1.5]
     at (567.5666, 740.7435) {ipe disk};
  \pic[ipe mark scale=1.5]
     at (576.7524, 740.7435) {ipe disk};
  \pic[ipe mark scale=1.5]
     at (581.3453, 736.1506) {ipe disk};
  \pic[ipe mark scale=1.5]
     at (581.3453, 726.9647) {ipe disk};
  \pic[ipe mark scale=1.5]
     at (576.7524, 722.3718) {ipe disk};
  \pic[ipe mark scale=1.5]
     at (567.5666, 722.3718) {ipe disk};
  \pic[ipe mark scale=1.5]
     at (558.3807, 722.3718) {ipe disk};
  \pic[ipe mark scale=1.5]
     at (549.1949, 722.3718) {ipe disk};
  \pic[ipe mark scale=1.5]
     at (544.602, 726.9647) {ipe disk};
  \pic[ipe mark scale=1.5]
     at (544.602, 736.1506) {ipe disk};
  \pic[ipe mark scale=1.5]
     at (540.0091, 722.3718) {ipe disk};
  \pic[ipe mark scale=1.5]
     at (530.8233, 722.3718) {ipe disk};
  \pic[ipe mark scale=1.5]
     at (521.6374, 722.3718) {ipe disk};
  \pic[ipe mark scale=1.5]
     at (512.4516, 722.3718) {ipe disk};
  \pic[ipe mark scale=1.5]
     at (507.8587, 726.9647) {ipe disk};
  \pic[ipe mark scale=1.5]
     at (503.2658, 722.3718) {ipe disk};
  \pic[ipe mark scale=1.5]
     at (507.8587, 736.1506) {ipe disk};
  \pic[ipe mark scale=1.5]
     at (494.08, 722.3718) {ipe disk};
  \pic[ipe mark scale=1.5]
     at (484.8942, 722.3718) {ipe disk};
  \pic[ipe mark scale=1.5]
     at (489.4871, 717.7789) {ipe disk};
  \pic[ipe mark scale=1.5]
     at (475.7083, 722.3718) {ipe disk};
  \pic[ipe mark scale=1.5]
     at (466.5225, 722.3718) {ipe disk};
  \pic[ipe mark scale=1.5]
     at (471.1154, 726.9647) {ipe disk};
  \pic[ipe mark scale=1.5]
     at (471.1154, 736.1506) {ipe disk};
  \pic[ipe mark scale=1.5]
     at (457.3367, 722.3718) {ipe disk};
  \pic[ipe mark scale=1.5]
     at (448.1509, 722.3718) {ipe disk};
  \pic[ipe mark scale=1.5]
     at (452.7438, 717.7789) {ipe disk};
  \pic[ipe mark scale=1.5]
     at (438.965, 722.3718) {ipe disk};
  \pic[ipe mark scale=1.5]
     at (429.7792, 722.3718) {ipe disk};
  \pic[ipe mark scale=1.5]
     at (434.3721, 726.9647) {ipe disk};
  \pic[ipe mark scale=1.5]
     at (434.3721, 736.1506) {ipe disk};
  \pic[ipe mark scale=1.5]
     at (420.5934, 722.3718) {ipe disk};
  \pic[ipe mark scale=1.5]
     at (416.0005, 717.7789) {ipe disk};
  \pic[ipe mark scale=1.5]
     at (416.0005, 708.5931) {ipe disk};
  \pic[ipe mark scale=1.5]
     at (420.5934, 704.0002) {ipe disk};
  \pic[ipe mark scale=1.5]
     at (429.7792, 704.0002) {ipe disk};
  \pic[ipe mark scale=1.5]
     at (438.965, 704.0002) {ipe disk};
  \pic[ipe mark scale=1.5]
     at (448.1509, 704.0002) {ipe disk};
  \pic[ipe mark scale=1.5]
     at (452.7438, 708.5931) {ipe disk};
  \pic[ipe mark scale=1.5]
     at (457.3367, 704.0002) {ipe disk};
  \pic[ipe mark scale=1.5]
     at (466.5225, 704.0002) {ipe disk};
  \pic[ipe mark scale=1.5]
     at (475.7083, 704.0002) {ipe disk};
  \pic[ipe mark scale=1.5]
     at (484.8942, 704.0002) {ipe disk};
  \pic[ipe mark scale=1.5]
     at (489.4871, 708.5931) {ipe disk};
  \pic[ipe mark scale=1.5]
     at (494.08, 704.0002) {ipe disk};
  \pic[ipe mark scale=1.5]
     at (503.2658, 704.0002) {ipe disk};
  \pic[ipe mark scale=1.5]
     at (512.4516, 704.0002) {ipe disk};
  \pic[ipe mark scale=1.5]
     at (521.6374, 704.0002) {ipe disk};
  \pic[ipe mark scale=1.5]
     at (526.2304, 708.5931) {ipe disk};
  \pic[ipe mark scale=1.5]
     at (530.8233, 704.0002) {ipe disk};
  \pic[ipe mark scale=1.5]
     at (526.2304, 717.7789) {ipe disk};
  \pic[ipe mark scale=1.5]
     at (540.0091, 704.0002) {ipe disk};
  \pic[ipe mark scale=1.5]
     at (549.1949, 704.0002) {ipe disk};
  \pic[ipe mark scale=1.5]
     at (558.3807, 704.0002) {ipe disk};
  \pic[ipe mark scale=1.5]
     at (562.9737, 708.5931) {ipe disk};
  \pic[ipe mark scale=1.5]
     at (562.9737, 717.7789) {ipe disk};
  \pic[ipe mark scale=1.5]
     at (434.3721, 763.708) {ipe disk};
  \pic[ipe mark scale=1.5]
     at (434.3721, 772.8939) {ipe disk};
  \draw[shift={(434.372, 772.894)}, scale=0.5741]
    (0, 0)
     -- (8, 8)
     -- (248, 8)
     -- (256, 0)
     -- (256, -16)
     -- (248, -24)
     -- (-24, -24)
     -- (-32, -32)
     -- (-32, -48)
     -- (-24, -56);
  \draw[shift={(420.593, 740.743)}, scale=0.5741]
    (0, 0)
     -- (272, 0)
     -- (272, 0);
  \draw[shift={(576.752, 740.743)}, scale=0.5741]
    (0, 0)
     -- (8, -8)
     -- (8, -24)
     -- (0, -32)
     -- (-272, -32);
  \draw[shift={(420.593, 722.372)}, scale=0.5741]
    (0, 0)
     -- (-8, -8)
     -- (-8, -24)
     -- (0, -32)
     -- (240, -32)
     -- (240, -32);
  \draw[shift={(558.381, 704)}, scale=0.5741]
    (0, 0)
     -- (8, 8)
     -- (8, 24)
     -- (0, 32);
  \draw[shift={(562.974, 717.779)}, scale=0.5741]
    (0, 0)
     -- (8, 8);
  \draw[shift={(530.823, 704)}, scale=0.5741]
    (0, 0)
     -- (-8, 8)
     -- (-16, 0);
  \draw[shift={(526.23, 708.593)}, scale=0.5741]
    (0, 0)
     -- (0, 16)
     -- (-8, 24);
  \draw[shift={(526.23, 717.779)}, scale=0.5741]
    (0, 0)
     -- (8, 8);
  \draw[shift={(494.08, 704)}, scale=0.5741]
    (0, 0)
     -- (-8, 8)
     -- (-16, 0);
  \draw[shift={(489.487, 708.593)}, scale=0.5741]
    (0, 0)
     -- (0, 16)
     -- (-8, 24);
  \draw[shift={(489.487, 717.779)}, scale=0.5741]
    (0, 0)
     -- (8, 8);
  \draw[shift={(457.337, 704)}, scale=0.5741]
    (0, 0)
     -- (-8, 8)
     -- (-16, 0);
  \draw[shift={(452.744, 708.593)}, scale=0.5741]
    (0, 0)
     -- (0, 16)
     -- (-8, 24)
     -- (-8, 24);
  \draw[shift={(452.744, 717.779)}, scale=0.5741]
    (0, 0)
     -- (8, 8);
  \draw[shift={(429.779, 722.372)}, scale=0.5741]
    (0, 0)
     -- (8, 8)
     -- (16, 0);
  \draw[shift={(434.372, 726.965)}, scale=0.5741]
    (0, 0)
     -- (0, 16)
     -- (-8, 24);
  \draw[shift={(434.372, 736.151)}, scale=0.5741]
    (0, 0)
     -- (8, 8);
  \draw[shift={(448.151, 740.743)}, scale=0.5741]
    (0, 0)
     -- (8, 8)
     -- (16, 0);
  \draw[shift={(466.523, 740.743)}, scale=0.5741]
    (0, 0)
     -- (8, -8)
     -- (16, 0);
  \draw[shift={(466.523, 722.372)}, scale=0.5741]
    (0, 0)
     -- (8, 8)
     -- (16, 0);
  \draw[shift={(448.151, 759.115)}, scale=0.5741]
    (0, 0)
     -- (8, -8)
     -- (16, 0);
  \draw[shift={(452.744, 745.336)}, scale=0.5741]
    (0, 0)
     -- (0, 16);
  \draw[shift={(484.894, 759.115)}, scale=0.5741]
    (0, 0)
     -- (8, -8)
     -- (16, 0);
  \draw[shift={(484.894, 740.743)}, scale=0.5741]
    (0, 0)
     -- (8, 8)
     -- (16, 0);
  \draw[shift={(489.487, 745.336)}, scale=0.5741]
    (0, 0)
     -- (0, 16);
  \draw[shift={(521.637, 759.115)}, scale=0.5741]
    (0, 0)
     -- (8, -8)
     -- (16, 0);
  \draw[shift={(521.637, 740.743)}, scale=0.5741]
    (0, 0)
     -- (8, 8)
     -- (16, 0);
  \draw[shift={(526.23, 745.336)}, scale=0.5741]
    (0, 0)
     -- (0, 16);
  \draw[shift={(558.381, 759.115)}, scale=0.5741]
    (0, 0)
     -- (8, -8)
     -- (16, 0);
  \draw[shift={(558.381, 740.743)}, scale=0.5741]
    (0, 0)
     -- (8, 8)
     -- (16, 0);
  \draw[shift={(562.974, 745.336)}, scale=0.5741]
    (0, 0)
     -- (0, 16);
  \draw[shift={(540.009, 759.115)}, scale=0.5741]
    (0, 0)
     -- (8, 8)
     -- (16, 0);
  \draw[shift={(544.602, 763.708)}, scale=0.5741]
    (0, 0)
     -- (0, 16)
     -- (-8, 24);
  \draw[shift={(544.602, 772.894)}, scale=0.5741]
    (0, 0)
     -- (8, 8);
  \draw[shift={(503.266, 759.115)}, scale=0.5741]
    (0, 0)
     -- (8, 8)
     -- (16, 0);
  \draw[shift={(507.859, 763.708)}, scale=0.5741]
    (0, 0)
     -- (0, 16)
     -- (-8, 24);
  \draw[shift={(507.859, 772.894)}, scale=0.5741]
    (0, 0)
     -- (8, 8);
  \draw[shift={(466.523, 759.115)}, scale=0.5741]
    (0, 0)
     -- (8, 8)
     -- (16, 0);
  \draw[shift={(471.115, 763.708)}, scale=0.5741]
    (0, 0)
     -- (0, 16);
  \draw[shift={(475.708, 777.487)}, scale=0.5741]
    (0, 0)
     -- (-8, -8)
     -- (-16, 0);
  \draw[shift={(429.779, 759.115)}, scale=0.5741]
    (0, 0)
     -- (8, 8)
     -- (16, 0);
  \draw[shift={(434.372, 763.708)}, scale=0.5741]
    (0, 0)
     -- (0, 16);
  \draw[shift={(434.372, 763.708)}, scale=0.5741]
    (0, 0)
     -- (0, 16);
  \draw[shift={(503.266, 740.743)}, scale=0.5741]
    (0, 0)
     -- (8, -8);
  \draw[shift={(512.452, 740.743)}, scale=0.5741]
    (0, 0)
     -- (-8, -8)
     -- (-8, -24)
     -- (0, -32);
  \draw[shift={(507.859, 726.965)}, scale=0.5741]
    (0, 0)
     -- (-8, -8);
  \draw[shift={(471.115, 726.965)}, scale=0.5741]
    (0, 0)
     -- (0, 16);
  \draw[shift={(540.009, 722.372)}, scale=0.5741]
    (0, 0)
     -- (8, 8)
     -- (16, 0);
  \draw[shift={(544.602, 726.965)}, scale=0.5741]
    (0, 0)
     -- (0, 16)
     -- (-8, 24);
  \draw[shift={(549.195, 740.743)}, scale=0.5741]
    (0, 0)
     -- (-8, -8);
  \draw[shift={(549.195, 740.743)}, scale=0.5741]
    (0, 0)
     -- (-8, -8);
  \pic[ipe mark scale=1.5]
     at (381.3445, 777.4868) {ipe disk};
  \pic[ipe mark scale=1.5]
     at (215.9999, 703.9996) {ipe disk};
  \node[ipe node]
     at (290.868, 682.14) {(b)};
  \node[ipe node]
     at (491.519, 682.14) {(c)};
\end{tikzpicture}
\caption{An example of an elementary wall (a), a $1$-subdivided wall (b), and the line graph of a $1$-subdivided wall (c).}\label{fig:walls and their line graphs}
\end{figure}

\begin{lemma}\label{lemma:H is not any H}
The class of balanced complete bipartite graphs
and, for all $q\ge 0$, the class of $q$-subdivided walls and the class of their line graphs, are $(\tw,\omega)$-unbounded.
\end{lemma}

\begin{proof}
It is well known that the minimum degree of a graph is a lower bound on its treewidth (see, e.g.,~\cite{MR2829452}).
Hence, $\tw(K_{n,n})\ge n$ and clearly, since $K_{n,n}$ is bipartite, we have $\omega(K_{n,n}) = 2$.
    We conclude that the class of balanced complete bipartite graphs is $(\tw,\omega)$-unbounded.

    The class of elementary walls has unbounded treewidth (see, e.g.,~\cite{MR3451044}), and, since the treewidth of a graph $G$ is at least as large as the treewidth of any of its minors (see, e.g.,~\cite{MR1647486}), so is the class of $q$-subdivided walls for any $q\ge 0$.
    Furthermore, since $\tw(L(G)) \geq \frac{1}{2} (\tw(G) + 1)-1$, as shown by Harvey and Wood~\cite{MR3820302}, the class of line graphs of $q$-subdivided walls also has unbounded treewidth. The clique number of each graph in these classes is bounded by~$3$, and hence all these classes are indeed $(\tw,\omega)$-unbounded.
\end{proof}

\Cref{lemma:H is not any H} implies the following necessary conditions for $(\tw, \omega)$-boundedness of a graph class excluding a single graph $H$ with respect to some graph containment relation (in particular, with respect to one of the six relations considered in this paper).

\begin{corollary}\label{H is in the three classes}
    Let $\subseteq$ be any graph containment relation, and let $H$ be a graph such that the class of graphs excluding $H$ with respect to relation $\subseteq$ is $(\tw, \omega)$-bounded.
    Then $H$ is in relation $\subseteq$ with some balanced complete bipartite graph, with some $q$-subdivided wall for each $q \geq 0$, and with the line graph of some $q'$-subdivided wall for each $q' \geq 0$.
\end{corollary}

\section{Forbidding an induced subgraph or an induced topological minor}\label{sec:isitm}

We first consider graph classes excluding a graph $H$ as in induced subgraph or as an induced topological minor.
The following characterization of $(\tw,\omega)$-bounded graph classes excluding a single forbidden induced subgraph is derived using
\cref{H is edgeless,H is in the three classes}.

\begin{theorem}\label{thm:forbidden induced subgraph}
Let $H$ a graph. Then, the class of $H$-free graphs is $(\tw,\omega)$-bounded if and only if one of the following conditions holds.
\begin{enumerate}
\item $H \subseteq_{is} P_3$ with a binding function $f(k) = k-1$.
\item $H$ is edgeless with a binding function $f(k) = R(k+1,|V(H)|)-2$.
\end{enumerate}
\end{theorem}

\begin{proof}
If $H$ is edgeless, then \cref{H is edgeless} applies.
If $H \subseteq_{is} P_{3}$, then every $H$-free graph $G$ is $P_{3}$-free and $G$ is a disjoint union of complete graphs.
Thus, $\tw(G) = \omega(G)-1$ in this case.

Suppose now that $H$ is neither edgeless nor an induced subgraph of $P_3$ and that the class of $H$-free graphs is $(\tw,\omega)$-bounded.
By \cref{H is in the three classes}, $H$ is an induced subgraph of some complete bipartite graph and also an induced subgraph of the line graph of some elementary wall.
In particular, $H$ must be isomorphic to a complete bipartite graph $K_{p,q}$ with $1 \leq p \leq q$ (note that we must have $p,q\ge 1$ since $H$ is not edgeless).
Furthermore, since line graphs of elementary walls are $\{\text{claw}, C_4\}$-free (or, equivalently, $\{K_{1,3},K_{2,2}\}$-free), we infer that $H$ must be isomorphic to either $K_{1,1}$ or $K_{1,2}$.
Thus, $H$ is an induced subgraph of $P_{3}$, a contradiction.
\end{proof}

A \textit{cut-vertex} in a connected graph $G$ is a vertex whose removal disconnects the graph. A \textit{block} of a graph is a maximal connected subgraph without cut-vertices. A \textit{block-cactus graph} is a graph every block of which is a cycle or a complete graph. In her PhD thesis~\cite{hartinger2017nove}, Hartinger proved that a graph is
$K_4^-$-induced-minor-free if and only if $G$ is a block-cactus graph. The same approach actually shows that these properties are also equivalent to excluding $K_4^-$ as an induced topological minor.

\begin{lemma}\label{diamond-itm-free graphs are block-cactus-graphs}
Let $G$ be a graph. Then, the following conditions are equivalent:
\begin{enumerate}
    \item $G$ is $K_4^-$-induced-minor-free.
    \item $G$ is $K_4^-$-induced-topological-minor-free.
    \item $G$ is a block-cactus graph.
\end{enumerate}
\end{lemma}

\begin{proof}
Since every induced topological minor in $G$ is also an induced minor, $G$ is $K_4^-$-induced-topological-minor-free
if it is $K_4^-$-induced-minor-free.

Suppose that $G$ is $K_4^-$-induced-topological-minor-free and that $G$ is not a block-cactus graph.
We first show that $G$ contains a hole.
Suppose not. Then $G$ is a chordal $K_4^-$-free graph and thus a block graph (see~\cite{MR175113}), that is, a graph every block of which is a complete graph.
But then, $G$ is a block-cactus graph, a contradiction.
Hence, $G$ must contain a hole $C$, and in particular there exists some block $B$ of $G$ such that $V(C)\subseteq V(B)$.
Since $B$ is connected but not a cycle, there exists a vertex $x \in V(B) \setminus V(C)$ with a neighbor in $V(C)$.
If $|N(x) \cap V(C)| \geq 2$, it is easy to see that $G$ contains a subdivision of $K_4^-$ as an induced subgraph, a contradiction.
Thus, $|N(x) \cap V(C)| = 1$ and every vertex in $V(B)\setminus V(C)$ has at most one neighbor in $C$.
Now, take a vertex $z\in  V(B) \setminus V(C)$ which has a neighbor $v \in V(C)$ such that $z$ minimizes the length of a shortest path $P$ between $z$ and $C$ not containing $v$.
We know that $P$ must exist since $B$ has no cut-vertex. Also, we may assume that $v$ has no other neighbor in $P$; otherwise we could replace $z$ with this vertex and get a shorter path.
Let $v' \in V(C)$ be the vertex of $P$ in $V(C)\setminus \{v\}$ and $z'$ be the neighbor of $v'$ in $P$.
Recall that $z'$ has only one neighbor in $V(C)$.
Using a similar argument as for $v$, we may assume that $v'$ has no other neighbor in $P$.
The minimality of $P$ implies that the internal vertices of $P$ do not have a neighbor in $C$.
Hence, $G[V(C) \cup V(P)]$ is a subdivision of $K_4^-$, a contradiction.
This shows that every $K_4^-$-induced-topological-minor-free graph is a block-cactus graph.

Finally, let $G$ be a block-cactus graph, and let $H$ be an induced minor of $G$. It is not difficult to see that the class of block-cactus graphs
is closed under vertex deletions and edge contractions. Therefore, $H$ is also a block-cactus graph. Since $K_4^-$ is not a block-cactus graph,
$H$ cannot be isomorphic to $K_4^-$. Therefore, $G$ is $K_4^-$-induced-minor-free.
\end{proof}

\begin{lemma}\label{lem:block-cactus-graphs}
The class of block-cactus graphs is $(\tw,\omega)$-bounded with a binding function $f(k) = \max\{k-1,2\}$.
\end{lemma}
\begin{proof}
The treewidth of a graph $G$ is the maximum treewidth of its blocks (see, e.g.,~\cite{MR1647486}).
Since the treewidth of a complete graph of order $k$ is $k-1$ and the treewidth of a cycle is two, the result follows.
\end{proof}

\begin{theorem}\label{thm:itm-free}
Let $H$ be a graph. Then, the class of $H$-induced-topological-minor-free graphs is $(\tw,\omega)$-bounded if and only if
one of the following conditions holds.
\begin{enumerate}
\item $H \subseteq_{is} C_3$ or $H \subseteq_{is} C_4$, in which case a binding function is $f(k) = k-1$.
\item $H \cong K_4^-$, in which case a binding function is $f(k) = \max\{k-1,2\}$.
\item $H$ is edgeless, in which case a binding function is $f(k) = R(k+1,|V(H)|)-2$.
\end{enumerate}
\end{theorem}

\begin{proof}
If $H$ is edgeless, then \cref{H is edgeless} applies.
If $H \subseteq_{is} C_3$ or $H \subseteq_{is} C_4$, then $H\subseteq_{itm}C_{4}$.
Hence, by \cref{observation:chordal}, the class of $H$-induced-topological-minor-free graphs is a subclass of the class of chordal graphs, and thus \cref{thm:chordal} applies.
If \hbox{$H \cong K_4^-$}, then according to \cref{diamond-itm-free graphs are block-cactus-graphs} the class of $H$-induced-topological-minor-free graphs is the class of block-cactus graphs, and~\cref{lem:block-cactus-graphs} applies.

For the converse direction, suppose that
$H \nsubseteq_{is} C_3$, $H \nsubseteq_{is} C_4$, $H \ncong K_4^-$, $H$ is not edgeless, and that the class of $H$-induced-topological-minor-free graphs is $(\tw,\omega)$-bounded.
By \cref{H is in the three classes}, $H$ is an induced topological minor of some complete bipartite graph and an induced topological minor of the line graph of some \hbox{$1$-subdivided} wall. Since the line graph of every \hbox{$1$-subdivided} wall is planar, subcubic, and claw-free, $H$ must also be planar, subcubic, and claw-free.
Furthermore, since $H$ is an induced topological minor of some complete bipartite graph, we must have $H\subseteq_{itm} K_{2,3}$, since otherwise
either $H$ would not be planar or it would not be subcubic.
Finally, claw-freeness implies that $H\in \{P_2, P_3, C_3, C_4, K_4^-\}$, a contradiction.
\end{proof}

\section{Forbidding an induced minor}\label{sec:im}

We now turn to graph classes excluding a single graph $H$ as an induced minor. Given a graph $G$, we denote by $\h(G)$ the \emph{Hadwiger number of $G$}, defined as the largest value of $p$ such that $K_p$ is a minor of $G$ (see~\cite{MR779891}).
We first develop some sufficient conditions for when sufficiently large Hadwiger number implies large clique number and then apply these results to characterize the graphs $H$ such that the class of $H$-induced-minor-free graphs is $(\tw,\omega)$-bounded.

\subsection{A detour: Hadwiger number versus clique number}

In \cref{lemma:Kn-1 im-free is (h omega)-bounded,lemma:W_4 im-free} we show that excluding either a complete graph minus an edge or a $4$-wheel as an induced minor results in an $(\h,\omega)$-bounded graph class with a linear binding function.

\begin{theorem}\label{lemma:Kn-1 im-free is (h omega)-bounded}
For each $p\ge 2$, the class of $K_{p}^-$-induced-minor-free graphs is $(\h,\omega)$-bounded with a binding function $f(k) = \max\{2p-4,k\}$.
\end{theorem}

\begin{proof}
Fix $p\ge 2$ and $k\in \mathbb{N}$, and let $G$ be a $K_{p}^-$-induced-minor-free graph with $\omega(G) = k$.
Let $q = \max\{2p-4,k\}+1$. We want to show that $G$ contains no $K_q$ as a minor. Suppose for a contradiction that $G$ contains $K_q$ as a minor. Fix a minor model $M = (X_u: u\in V(K_q))$ of $K_q$ in $G$ such that the total number of vertices in the bags, that is, the sum $\sum_{u\in V(K_q)}|X_u|$, is minimized.

If for all $u\in V(K_q)$ we have $|X_u| = 1$, then the set $\bigcup_{u\in V(K_q)}X_u$ is a clique in $G$, implying that $\omega(G)\ge |V(K_q)| = q \ge k+1$, a contradiction.
Therefore, there exists some $u\in V(K_q)$ such that $|X_u| \ge 2$. Furthermore, note that for every vertex $y\in X_u$ there exists a vertex $v(y)$ of $K_q-u$ such that $y$ has no neighbors in $X_{v(y)}$, since otherwise replacing the bag $X_u$ with $\{y\}$ would result in a minor model of $K_q$ smaller than $M$.
Since $|X_u|\ge 2$ and the subgraph of $G$ induced by $X_u$ is connected, there exists a vertex $x\in X_u$ such that the subgraph of $G$ induced by $X_u\setminus \{x\}$ is connected. (For example, take $x$ to be a leaf of a spanning tree of $G[X_u]$.)

Let $Z$ be the set of vertices $z\in V(K_q)\setminus \{u\}$ such that $x$ has a neighbor in $X_z$.
Suppose first that $|Z|\ge (q-1)/2$. Recall that $X_{v(x)}$ is a bag in which $x$ has no neighbor.
In particular, $v(x)\neq u$ and $v(x)\not\in Z$. Then, the bags from $(X_z: z\in Z)$ along with $\{x\}$ and $X_{v(x)}$ form an induced minor model of $K_{|Z|+2}^-$.
Since $|Z|+2 \ge (q-1)/2+2\ge (2p-4)/2+2= p$, we obtain a contradiction with the fact that $G$ is $K_{p}^-$-induced-minor-free.

Finally, suppose that $|Z|<(q-1)/2$.
The minimality of $M$ implies that $Z$ is nonempty and for some $w\in Z$ we have $\left(\bigcup_{v \in X_{w}} N(v)\right) \cap X_u = \{x\}$.
Let $Z' = V(K_q)\setminus(Z\cup \{u\})$.
Note that for every vertex $z\in Z'$ there exists an edge from $X_z$ to $X_u\setminus\{x\}$.
Since $|Z|+|Z'| = q-1$ and $|Z|<(q-1)/2$, we have $|Z'|\ge (q-1)/2$.
Furthermore, $w\in Z$ and hence $w\not\in Z'$.
Thus, the bags from $(X_z: z\in Z')$ along with $X_u\setminus \{x\}$ and $X_{w}$ form an induced minor model of $K_{|Z'|+2}^-$, leading again to a contradiction with the fact that $G$ is $K_{p}^-$-induced-minor-free.
\end{proof}

Similar but more involved arguments show that large Hadwiger number implies large clique number also in the class of $W_4$-induced-minor-free graphs.
In the proof of the next theorem we will need the following standard notion: A vertex $u$ in a graph $G$ is said to be \emph{universal} if it is adjacent to all other vertices of $G$, that is, $d_G(u) = |V(G)|-1$.

\begin{theorem}\label{lemma:W_4 im-free}
The class of $W_4$-induced-minor-free graphs is $(\h,\omega)$-bounded with a binding function $f(k) = k+5$.
\end{theorem}

\begin{proof}
Fix a positive integer $k$, and let $G$ be a $W_4$-induced-minor-free graph with $\omega(G) = k$.
Let $q = k+5$ (note that $q \geq 6$) and $F$ be the graph $K_q^-$.
We claim that $G$ does not contain $F$ as an induced minor.
We denote by $U \subset V(F)$ the set of universal vertices in $F$.
To derive a contradiction, suppose that $G$ contains $F$ as an induced minor, and fix an induced minor model $M = (X_u : u \in V(F))$ of $F$ in $G$ such that the size of $\bigcup_{u \in U} X_u$ is minimized. We will refer to this condition as \emph{property~$(\mystar)$}.
We denote by $x$ and $y$ the two nonadjacent vertices in $F$.
It is clear that if for all $u \in U$ we have $|X_u| = 1$, then the set $\bigcup_{u \in U} X_u$ is a clique in $G$, a contradiction since $|U| = q-2 > k= \omega(G)$.
Hence, there exists a vertex $u \in U$ such that $|X_u| \geq 2$.

Partition the bag $X_u$ arbitrarily into two nonempty bags $X_{u_1}$ and $X_{u_2}$, both inducing a connected subgraph in $G$.
(For example, we can take $\ell$ to be a leaf of a spanning tree of $G[X_u]$ and set
$X_{u_1} = \{\ell\}$ and $X_{u_2} = X_u \setminus \{\ell\}$.)
Let $M'$ be the collection of bags obtained from $M$ by removing the bag $X_u$ and adding the bags $X_{u_1}$
and $X_{u_2}$. Let $F'$ be the graph obtained from the subgraph of $G$ induced by the union of bags in $M'$
by contracting each of the bags in $M'$ into a single vertex.
Note that the vertex set of $F'$ is $(V(F)\setminus\{u\})\cup\{u_1,u_2\}$ and that $M'$ is an induced minor model of $F'$ in $G$.
In particular, $F'$ is an induced minor of $G$. Notice that $u_1$ and $u_2$ are adjacent in $F'$.
Let $U' = U\setminus\{u\}$ and observe that $U'\subseteq V(F')$.
Note that $d_{F'}(u_1) \geq 2$; otherwise $u_1$ would only be adjacent to $u_2$, and thus we could replace $X_u$ with $X_{u_2}$ in $M$ to obtain an induced minor model of $F$ in $G$ that would contradict the fact that $M$ satisfies property~$(\mystar)$.
For the same reason, $d_{F'}(u_2) \geq 2$.

Suppose first that $d_{F'}(u_1) = 2$. Let $v$ be the neighbor of $u_1$ different from $u_2$.
If $v = x$, then we could redefine $X_{u} = X_{u_2}$ and $X_{x} = X_{x} \cup X_{u_1}$ in $M$ to obtain an induced minor model of $F$ in $G$ showing that $M$ does not respect property~$(\mystar)$. Thus, $v\neq x$. Similarly, $v\neq y$. Consequently, $v\in U'$.
The fact that $M$ satisfies property~$(\mystar)$ also implies that $v$ is not adjacent to $u_2$.
Since $|U'| = |U|-1 = q-3 > 2$, there is a vertex $w\in U'\setminus\{v\}$.
Note that $w$ is not adjacent to $u_1$ and hence must be adjacent to $u_2$.
We obtain that $\{x,u_2,y,v\}$ induces a $C_4$ in $F'$ and $\{x,u_2,y,v\} \subseteq N(w)$.
Therefore, $G$ contains $W_4$  as an induced minor, a contradiction.
Thus, we have $d_{F'}(u_1) \ge 3$.
By symmetry, we also have $d_{F'}(u_2) \ge 3$.

For $i \in \{1,2\}$, let $A_i$ be the set of vertices in $U'$ adjacent to $u_i$.
By symmetry, it suffices to consider the following two cases depending on $A_1$ and $A_2$.
\begin{enumerate}[{Case} 1:]
    \item $A_1 \nsubseteq A_2$ and $A_2 \nsubseteq A_1$.\\
    Let $v \in A_1 \setminus A_2$ and $w \in A_2 \setminus A_1$.
    Notice that $v$ and $w$ are adjacent, and therefore $\{v,u_1,u_2,w\}$ induces a $C_4$ in $F'$.
    Suppose first that $u_1$ is adjacent to neither $x$ nor $y$.
    Then $u_2$ is adjacent to both $x$ and $y$.
    Furthermore, since $d_{F'}(u_1) \ge 3$, vertex $u_1$ must have a neighbor $z \in U' \setminus \{v,w\}$.
    Hence, every vertex in $\{v,u_1,u_2,w\}$ has a neighbor in the set $\{x,y,z\}$. Since
    $\{x,y,z\}$ induces a connected subgraph of $F'$, we infer that $W_4$ is an induced minor of $F'$ and thus of $G$, a contradiction.
    A similar conclusion is obtained if $u_2$ is adjacent to neither $x$ nor $y$.
    We may thus assume that $u_1$ is adjacent to either $x$ or $y$, and the same for $u_2$.
    Since $|U'| = |U|-1 = q-3 \geq 3$, there is a vertex $z\in U'\setminus\{v,w\}$.
    Again, since $\{x,y,z\}$ induces a connected subgraph of $F'$, we conclude that
    $W_4$ is an induced minor of $F$ and thus of $G$, a contradiction.
    See \cref{fig:induced C4 in U u1 u2} for an illustration.

  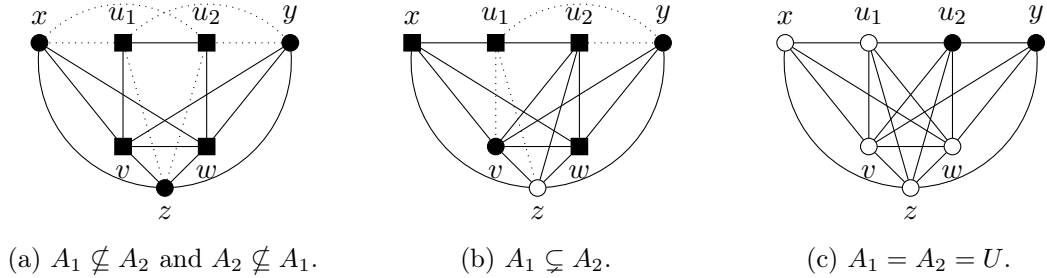
\begin{figure}[tbh]
  \centering
  \captionsetup[subfigure]{justification=centering}
  \begin{subfigure}[t]{0.3\textwidth}
  \centering
\begin{tikzpicture}[scale=0.55]
\tikzset{bag/.style={circle,
				inner sep=0pt,
				text width=4mm,
				align=center,
				draw=black,
				transform shape,
				fill=none}}
  \node[bag,label={[]above:$x$},fill=black] (x) at (-3,3) {};
  \node[bag,label={[]above:${u_1}$},regular polygon,regular polygon sides=4,fill=black,scale=0.7] (u1) at (-1,3) {};
  \node[bag,label={[]above:${u_2}$},regular polygon,regular polygon sides=4,fill=black,scale=0.7] (u2) at (1,3) {};
  \node[bag,label={[]above:$y$},fill=black] (y) at (3,3) {};
  \node[bag,label={[]below:$v$},regular polygon,regular polygon sides=4,fill=black,scale=0.7] (w) at (-1,.5) {};
  \node[bag,label={[]below:$w$},regular polygon,regular polygon sides=4,fill=black,scale=0.7] (w') at (1,.5) {};
  \node[bag,label={[]below:$z$},fill=black] (z) at (0,-.5) {};

  \draw (w) to (z) to (w');
  \draw (x) to [bend right=45] (z);
  \draw (y) to [bend left=45] (z);
  \draw (w) to (w');
  \draw (x) to (w);
  \draw (x) to (w');
  \draw (y) to (w);
  \draw (y) to (w');

  \draw (u1) to (u2);

  \draw (u1) to (w);

  \draw (u2) to (w');

  \draw[dotted] (u1) to (x);
  \draw[dotted] (x) to [bend left=45] (u2);
  \draw[dotted] (u1) to [bend left=45] (y);
  \draw[dotted] (y) to (u2);
  \draw[dotted] (u1) to (z);
  \draw[dotted] (z) to (u2);
\end{tikzpicture}
    \caption{$A_1 \nsubseteq A_2$ and $A_2 \nsubseteq A_1$.}\label{fig:induced C4 in U u1 u2}
  \end{subfigure}
  \begin{subfigure}[t]{0.3\textwidth}
    \centering
\begin{tikzpicture}[scale=0.55]
\tikzset{bag/.style={circle,
				inner sep=0pt,
				text width=4mm,
				align=center,
				draw=black,
				transform shape,
				fill=none}}
  \node[bag,label={[]above:$x$},regular polygon,regular polygon sides=4,fill=black,scale=0.7] (x) at (-3,3) {};
  \node[bag,label={[]above:${u_1}$},regular polygon,regular polygon sides=4,fill=black,scale=0.7] (u1) at (-1,3) {};
  \node[bag,label={[]above:${u_2}$},regular polygon,regular polygon sides=4,fill=black,scale=0.7] (u2) at (1,3) {};
  \node[bag,label={[]above:$y$},fill=black] (y) at (3,3) {};
  \node[bag,label={[]below:$v$},fill=black] (v) at (-1,.5) {};
  \node[bag,label={[]below:${w}$},regular polygon,regular polygon sides=4,fill=black,scale=0.7] (w) at (1,.5) {};
  \node[bag,label={[]below:${z}$}] (z) at (0,-.5) {};

  \draw (v) to (z);
  \draw (z) to (w);
  \draw (x) to [bend right=45] (z);
  \draw (y) to [bend left=45] (z);
  \draw (v) to (w);
  \draw (x) to (v);
  \draw (x) to (w);
  \draw (y) to (v);
  \draw (y) to (w);

  \draw (x) to (u1);
  \draw[dotted] (u1) to [bend left=45] (y);
  \draw (u1) to (u2);

  \draw (u2) to (v);
  \draw (u2) to (w);
  \draw (u2) to (z);
  \draw[dotted] (u1) to (v);

  \draw[dotted] (u1) to (z);
  \draw[dotted] (u2) to (y);
\end{tikzpicture}
    \caption{$A_1 \subsetneq A_2$.}\label{fig:N_U(u1) subset N_U(u2)}
  \end{subfigure}
  \begin{subfigure}[t]{0.3\textwidth}
    \centering
\begin{tikzpicture}[scale=0.55]
\tikzset{bag/.style={circle,
				inner sep=0pt,
				text width=4mm,
				align=center,
				draw=black,
				transform shape,
				fill=none}}
  \node[bag,label={[]above:$x$},fill=none] (x) at (-3,3) {};
  \node[bag,label={[]above:${u_1}$},fill=none] (u1) at (-1,3) {};
  \node[bag,label={[]above:${u_2}$},fill=black] (u2) at (1,3) {};
  \node[bag,label={[]above:$y$},fill=black] (y) at (3,3) {};
  \node[bag,label={[]below:$v$},fill=none] (v) at (-1,.5) {};
  \node[bag,label={[]below:${w}$},fill=none] (w) at (1,.5) {};
  \node[bag,label={[]below:${z}$},fill=none] (z) at (0,-.5) {};

  \draw (z) to (u1);
  \draw (z) to (u2);
  \draw (u1) to (w);
  \draw (u2) to (v);
  \draw (x) to (u1);
  \draw (y) to (u2);
  \draw (x) to (v);
  \draw (x) to (w);
  \draw (y) to (v);
  \draw (y) to (w);
  \draw (z) to (v);
  \draw (z) to (w);
  \draw (u1) to (u2);
  \draw (v) to (w);
  \draw (u1) to (v);
  \draw (u2) to (w);
  \draw (x) to [bend right=45] (z);
  \draw (y) to [bend left=45] (z);
\end{tikzpicture}
    \caption{$A_1 = A_2 = U$.}\label{fig:u1 and u2 are fully connected to U}
  \end{subfigure}

 \caption{Representation of the different cases considered in the proof of \cref{lemma:W_4 im-free}. The induced minor contains all plain edges and is a subgraph of the graph induced by plain and dotted edges. Black squared vertices induce a $C_4$ and black round vertices are merged into a single vertex.}
\end{figure}

   \item $A_1 \subseteq A_2$.\\
Necessarily, $A_2 = U'$, and hence $u_2$ cannot be adjacent to both $x$ and $y$, otherwise $M$ would not satisfy property~$(\mystar)$.
Without loss of generality, assume that $u_2$ is not adjacent to $x$.
Then $u_1$ is adjacent to $x$.

    Suppose first that $A_1$ is a proper subset of $A_2$.
    Then there exists a vertex $w\in A_2\setminus A_1$.
    Note that the vertices $\{x,u_1,u_2,w\}$ induce a $C_4$ in $F'$.
    We claim that $A_1 = \emptyset$.
    Indeed, suppose for a contradiction that there exists a vertex $v\in A_1$.
    Then $v\neq w$ and $v$ is universal in $F'$. Therefore, $F'$ contains an induced copy of $W_4$ in $F'$ with vertex set $\{x,u_1,u_2,w,v\}$; in particular, $W_4$ is an induced minor of $G$, a contradiction.
    Thus, $A_1 = \emptyset$, as claimed.
    This means that $u_1$ does not have any neighbors in $U'$, and hence using the inequality  $d_{F'}(u_1)\ge 3$ we infer that $N_{F'}(u_1) = \{u_2,x,y\}$.
    Recall that $|U'| \geq 3$.
    Choose any vertex $z \in U' \setminus \{w\}$.
    Note that every vertex in $F'$ is adjacent to either $y$ or $z$. In particular, since $\{y,z\}\cap \{x,u_1,u_2,w\}= \emptyset$ and
    $\{y,z\}$ induces a connected subgraph of $F'$, we conclude that $W_4$ is an induced minor of $F$, and thus of $G$, a contradiction.
    See \cref{fig:N_U(u1) subset N_U(u2)} for an illustration.

    We may thus assume that $A_1 = A_2$.
    The fact that $M$ satisfies property~$(\mystar)$ implies that $u_1$ cannot be adjacent to both $x$ and $y$. Thus, since $u_1$ is adjacent to $x$, it is not adjacent to $y$. Consequently, $u_2$ is adjacent to $y$.
    Now, observe that the graph obtained by contracting the edge $\{u_2,y\}$ in $F'$ is isomorphic to $F$.
    Hence, we can modify $M$ by redefining $X_{u} = X_{u_1}$ and $X_y := X_y \cup X_{u_2}$ and get a minor model of $F$.
    However, this implies that $M$ does not respect property~$(\mystar)$, a contradiction.
    See \cref{fig:u1 and u2 are fully connected to U} for an illustration.
  \end{enumerate}
  We conclude that $G$ is $K_q^-$-induced-minor-free, and following \cref{lemma:Kn-1 im-free is (h omega)-bounded} we obtain that $\h(G) \leq \max\{2q-5,k\} = 2k+1$.
\end{proof}

\subsection{Back to treewidth}

As explained by \citeauthor{MR3741534}~\cite{MR3741534} (and observed also in~\cite{MR3853110}), the following fact can be derived from the proof of Theorem 9 in~\cite{MR2901091}.

\begin{theorem}\label{thm:FH}
For every graph $F$ and every planar graph $H$, the class of graphs that are both $F$-minor-free and $H$-induced-minor-free has bounded treewidth.
\end{theorem}

Since excluding a complete graph as a minor is the same as excluding it as an induced minor, \cref{thm:FH} implies the following.

\begin{corollary}\label{lemma:Kn and planar H imf implies bounded tw}
For every positive integer $p$ and every planar graph $H$, the class of $\{K_{p},H\}$-induced-minor-free graphs has bounded treewidth.
\end{corollary}

Observe that no graph $G$ contains $K_{\eta(G)+1}$ as a minor (or, equivalently, as an induced minor).

\begin{corollary}\label{corollary:hadwiger-planar}
Let $H$ be a planar graph. The class of $H$-induced-minor-free graphs is $(\h,\omega)$-bounded if and only if it is $(\tw,\omega)$-bounded.
\end{corollary}

\begin{proof}
Suppose that the class of $H$-induced-minor-free graphs is $(\h,\omega)$-bounded, and let $f$ be a $(\h,\omega)$-binding function for the class.
Let $k\in \mathbb{N}$, and let $G$ be an $H$-induced-minor-free graph with $\omega(G) = k$.
Then $\h(G)\le f(k)$, that is, $G$ is $K_{f(k)+1}$-induced-minor-free.
By \cref{lemma:Kn and planar H imf implies bounded tw}, the treewidth of $G$ can be bounded from above by some constant $g(k)$ depending only on $k$.
Thus, $g$ is a $(\tw,\omega)$-binding function for the class.
\end{proof}

From \cref{lemma:Kn-1 im-free is (h omega)-bounded} we obtain that the class of $K_5^-$-induced-minor-free graphs is $(\h,\omega)$-bounded.
Since $K_5^-$ is planar, a direct application of \cref{corollary:hadwiger-planar} implies the following result.

\begin{corollary}\label{cor:K_5-}
The class of $K_5^-$-induced-minor-free graphs is $(\tw,\omega)$-bounded.
\end{corollary}

Similarly, since $W_4$ is planar, we can directly apply \cref{lemma:W_4 im-free,corollary:hadwiger-planar} and obtain the following result.

\begin{corollary}\label{cor:W4}
    The class of $W_4$-induced-minor-free graphs is $(\tw,\omega)$-bounded.
\end{corollary}

Our next result makes use of minimal separators. Given two nonadjacent vertices $u$ and $v$ in a graph $G$, a \emph{$u{,}v$-separator} in $G$ is a set $S$ of vertices such that $u$ and $v$ are in different connected components of $G-S$. A $u{,}v$-separator is \emph{minimal} if it does not contain any other $u{,}v$-separator. A \emph{minimal separator} in a graph $G$ is a minimal $u{,}v$-separator for some nonadjacent vertex pair $u,v$.
Given a graph $G$ and a set $S\subseteq V(G)$, an \emph{$S$-full component} of $G-S$ is a component $C$ of the graph $G-S$ such that every vertex in $S$ has a neighbor in $C$. The following lemma characterizing minimal separators is well known (see, e.g.,~\cite{MR2063679}).

\begin{lemma}\label{lem:separator}
A set $S$ of vertices in a graph $G$ is a minimal separator if and only if the graph $G-S$ has at least two $S$-full components.
\end{lemma}

\begin{theorem}[Skodinis \cite{MR1852483}]\label{bounded min sep implies bdd tw}
Let $s$ be a positive integer, and let $\mathcal{G}$ be the class of graphs in which all minimal separators have size at most $s$.
Then, $\mathcal{G}$ is $(\tw,\omega)$-bounded with a binding function $f(k) = \max \{k, 2s\}-1$.
\end{theorem}

Using \cref{bounded min sep implies bdd tw}, we infer our next result.

\begin{lemma}\label{K2n-im-free graphs have bdd tw}
  For every $q \in \mathbb{N}$, the class of $K_{2,q}$-induced-minor-free graphs is $(\tw,\omega)$-bounded with a
  binding function $f(k) = \max \{k, 2R(k+1,q) -2\}-1$.
\end{lemma}

\begin{proof}
Fix two positive integers $q$ and $k$, and let $G$ be a $K_{2,q}$-induced-minor-free graph with $\omega(G) = k$.
We claim that every minimal separator in $G$ has size at most $R(k+1,q) -1$. Suppose this is not the case, and let
$u$ and $v$ be two nonadjacent vertices in $G$ such that $|S| \geq R(k+1,q)$ for some
minimal $u{,}v$-separator $S$ in $G$. Since $|S| \geq R(k+1,q)$, Ramsey's theorem implies that $G[S]$ contains either a clique of size $k+1$ or
an independent set of size $q$. Since $\omega(G[S])\le \omega(G) = k$, we infer that $G[S]$ contains an independent set $I$ of size $q$.
 Let $C_u$ and $C_v$ denote the connected components of $G-S$ containing $u$ and $v$, respectively.
By the minimality of $S$, every vertex in $S$ has a neighbor in $C_u$ and a neighbor in $C_v$ (see, e.g.,~\cite{MR2063679}).
  But now, the sets  $V(C_u)$, $V(C_v)$, and $\{x\}$ for all $x\in I$ form the bags of an induced minor model of $K_{2,q}$ in $G$, a contradiction.
Therefore, every minimal separator in $G$ has size at most $R(k+1,q) -1$.
  Using \cref{bounded min sep implies bdd tw}, we obtain that $\tw(G) \leq
  \max \{k, 2R(k+1,q) -2\}-1$.
\end{proof}

\begin{remark}\label{remark:K2n}
The binding function given by \cref{K2n-im-free graphs have bdd tw} cannot be improved by means of improving the Ramsey number when restricted to the class of $K_{2,q}$-induced-minor-free graphs. Indeed, for every two positive integers $k$ and $q$, the least positive integer $N$ such that every $K_{2,q}$-induced-minor-free graph with at least $N$ vertices contains either a clique of size $k$ or an independent set of size $q$ equals the Ramsey number $R(k,q)$. This follows from~\cite[Theorem 2]{MR3202286};
the key observation is that there is a graph with $R(k,q)-1$ vertices having no clique of size $k$ and no independent set of size $q$, and every such graph is $K_{2,q}$-induced-minor-free.
\end{remark}

A graph $G$ is said to be \emph{$1$-perfectly orientable} if it has an orientation $D$ such that for every vertex $v\in V(G)$, the out-neighborhood of $v$ in $D$ is a clique in~$G$.
The class of $1$-perfectly orientable graphs is a common generalization of the classes of chordal graphs and circular-arc graphs.
While $1$-perfectly orientable graphs were studied in several papers (see, e.g.,~\cite{MR1244934,MR3647815,MR3853110}), their structure remains poorly understood.
\citeauthor{MR3853110} showed in~\cite{MR3853110} that the treewidth of every $1$-perfectly orientable planar graph is at most 21 and asked whether the class of $1$-perfectly orientable graphs is $(\tw,\omega)$-bounded. Since every $1$-perfectly orientable graph excludes $K_{2,3}$  as an induced minor (see~\cite{MR3647815}),  \cref{K2n-im-free graphs have bdd tw} answers their question in the affirmative.

\begin{corollary}\label{1po}
The class of $1$-perfectly orientable graphs is $(\tw,\omega)$-bounded with a binding function $f(k) = \max \{k, 2R(k+1,3) -2\}-1$.
\end{corollary}

\Cref{lemma:H is not any H} and \cref{cor:W4,cor:K_5-} lead to the following characterization.

\begin{theorem}\label{thm:necessary forbidden induced minors}
Let $H$ be a graph. Then, the class of $H$-induced-minor-free graphs is $(\tw,\omega)$-bounded if and only if one of the following conditions holds:
$H\subseteq_{is} W_4$, $H \subseteq_{is} K_5^-$, $H \subseteq_{is} K_{2,q}$ for some $q\in \mathbb{N}$, or $H \subseteq_{is} K_{2,q}^+$ for some $q\in \mathbb{N}$.
\end{theorem}

\begin{proof}
Suppose that the class of $H$-induced-minor-free graphs is $(\tw,\omega)$-bounded.
Since, by \cref{lemma:H is not any H}, the class of balanced complete bipartite graphs is $(\tw,\omega)$-unbounded,
$H$ must be an induced minor of some complete bipartite graph $K_{n,n}$. Let $M = (X_u : u \in H)$ be an induced minor model of $H$ in $K_{n,n}$.
We define two types of bags in $M$: the \emph{tiny bags} containing a single vertex and the \emph{large bags} containing at least $2$ vertices.
It is clear that the set of large bags corresponds to a clique in $H$, while the union of the tiny bags induces a complete bipartite subgraph of $K_{n,n}$.
Hence, $H \cong K_{p,q} \ast K_r$ for some $p,q,r\ge 0$, where $\ast$ represents the join of the two graphs, that is, the addition of all possible edges between vertices in $K_{p,q}$ and vertices in $K_r$. Without loss of generality, we assume that $p \leq q$.
Observe that $H$ needs to be planar; otherwise the class of $H$-induced-minor-free graphs would contain the class of elementary walls, which
by~\cref{lemma:H is not any H} is $(\tw,\omega)$-unbounded.
Hence, we can analyze the possible values for $p$, $q$, and $r$ that allow $H$ to be planar.
Let us first notice that $p \leq 2$, as otherwise $H$ would contain $K_{3,3}$ as a subgraph and would thus be planar.
Similarly, $r \leq 4$ since otherwise $H$ would contain $K_5$ as a subgraph. Also, it is easily observed that if $p=1$, then $H \cong K_{0,q} \ast K_{r+1}$, and similarly
if $q=1$, then $H \cong K_{p,0} \ast K_{r+1}$.
Hence, we may assume that $p \in \{0,2\}$ and $q \neq 1$.
Consider the following cases:
\begin{itemize}
  \item Case $r=4$:
    Then $p = q = 0$, otherwise $K_5 \subseteq_s H$.
    Hence, $H \cong K_4$.

  \item Case $r=3$:
    Then $p = 0$, otherwise $K_{3,3} \subseteq_s H$.
    If $q \geq 3$, then $K_{3,3} \subseteq_s H$, and thus $q \leq 2$.
    If $q = 0$, then $H \cong K_3$, and if $q=2$, then $H \cong K_5^-$.

  \item Case $r=2$:
    Then $p = 0$, otherwise $K_{3,3} \subseteq_s H$.
    This implies that $H \cong K_{2,q}^{+}$.

  \item Case $r=1$:
    If $p = 2$, then $q = 2$ (since otherwise $K_{3,3} \subseteq_s H$) and $H \cong W_4$.
    If $p = 0$, then $H \cong K_{1,q}$.

  \item Case $r=0$:
    Then $H$ is edgeless or $H\cong K_{2,q}$.
\end{itemize}
Thus, $H\subseteq_{is} W_4$, $H \subseteq_{is} K_5^-$, $H \subseteq_{is} K_{2,q}$, or $H \subseteq_{is} K_{2,q}^+$ for some $q\in \mathbb{N}$, as desired.

For the converse, suppose first that $H \subseteq_{is} K_{2,q}$ or $H \subseteq_{is} K_{2,q}^+$ for some $q\in \mathbb{N}$. It is not difficult to notice that $K_{2,q}^+$ is an induced minor of $K_{2,q+1}$, obtained by contracting one edge. From \cref{K2n-im-free graphs have bdd tw} it then follows that the class of $H$-induced-minor-free graphs is $(\tw,\omega)$-bounded. If $H\subseteq_{is} W_4$ or $H \subseteq_{is} K_5^-$,
then \cref{cor:W4,cor:K_5-} apply.
\end{proof}

\Cref{thm:necessary forbidden induced minors,corollary:(tw omega)-B implies chi-B} have the following consequence.

\begin{corollary}\label{chi-bounded-induced-minor}
Let $H$ be a graph such that $H\subseteq_{is} W_4$, $H \subseteq_{is} K_5^-$, $H \subseteq_{is} K_{2,q}$ for some $q\in \mathbb{N}$, or $H \subseteq_{is} K_{2,q}^+$ for some $q\in \mathbb{N}$. Then the class of $H$-induced-minor-free graphs is $\chi$-bounded.
\end{corollary}

To the best of our knowledge, $\chi$-boundedness of the classes of $H$-induced-minor-free graphs whenever $H$ is isomorphic to $W_4$, $K_5^-$,
$K_{2,q}$ for some $q\ge 3$, or $K_{2,q}^+$ for some $q\ge 3$, was not known prior to our work.

\begin{remark}
If $H$ is an induced subgraph of any of the above listed graphs but not isomorphic to any of them, then $\chi$-boundedness of the class of $H$-induced-minor-free graphs follows from results in the literature.
Indeed, in this case, $H$ is either an induced subgraph of $K_{1,q}$ for some $q\ge 3$ or an induced subgraph of $C_4$, $K_4^-$, or $K_4$.
If $H$ is an induced subgraph of $K_{1,q}$ for some $q\ge 3$, then
$\chi$-boundedness of the class of $H$-induced-minor-free graphs follows,
for example, from an easy application of Ramsey's theorem to the class of $K_{1,q}$-free graphs.
The cases when $H$ is an induced subgraph of $C_4$ and $K_4^-$
correspond, respectively, to the classes of chordal and block-cactus graphs (by~\cref{observation:chordal,diamond-itm-free graphs are block-cactus-graphs}).
In the former case, $\chi$-boundedness follows from the fact that chordal graphs are perfect. In the latter case, we can use the fact that block-cactus graphs have bounded clique-width~\cite{MR2536473}, which is a sufficient condition for $\chi$-boundedness~\cite{MR4125349,MR3350076}.
The case when $H$ is an induced subgraph of $K_4$ corresponds to the class of $K_4$-topological-minor-free graphs, and all such graphs are $3$-colorable~\cite{MR45371}.
The above arguments do not apply if $H$ is isomorphic to any of $W_4$, $K_5^-$, $K_{2,3}$, or $K_{2,3}^+$, as none of the resulting classes is contained in the class of perfect graphs or in any graph class of bounded chromatic number or bounded clique-width. (This can be seen using the results of~\cite{MR3741534}, the fact that complete graphs are $H$-induced-minor-free, and that odd cycles of length at least $5$ are $H$-induced-minor-free but not perfect.)
\end{remark}

\section{Forbidding a subgraph, a topological minor, or a minor}\label{sec:stmm}

We now complete our six dichotomy theorems by showing that known results on treewidth and graph minors imply characterizations of \hbox{$(\tw,\omega)$-bounded} graph classes excluding a single graph as either a subgraph, a topological minor, or a minor.
We start with a simple but useful observation.

\begin{lemma}\label{bounded treewidth}
Let $H$ be a graph, and let $\mathcal{G}$ be a graph class contained in the class of $H$-subgraph-free graphs. Then $\mathcal{G}$ is $(\tw,\omega)$-bounded if and only if $\mathcal{G}$ has bounded treewidth.
\end{lemma}

\begin{proof}
Any graph class having bounded treewidth is $(\tw,\omega)$-bounded.
For the converse direction, assume that $\mathcal{G}$ is $(\tw,\omega)$-bounded with a binding function $f$.
Note that no graph $G\in \mathcal{G}$ can have a clique of size $|V(H)|$, since otherwise $G$ would not be $H$-subgraph-free.
Hence, every graph $G\in \mathcal{G}$ satisfies $\omega(G)\in\{1,\ldots, k\}$, where $k = |V(H)|-1$, and thus the treewidth of $G$ is at most $\max\{f(1),\ldots, f(k)\}$.
\end{proof}

\begin{theorem}[Robertson and Seymour \cite{MR854606}]\label{planar minor implies bounded treewidth}
For every planar graph $H$, the class of $H$-minor-free graphs has bounded treewidth.
\end{theorem}

Recall that $\mathcal{S}$ is the class of graphs in which every connected component is either a path or a subdivided claw.

\begin{lemma}[\citeauthor{MR3280698}~\cite{MR3280698}]\label{lemma:kS_ell as minor <=> kS_ell as subgraph}
For every $H \in \mathcal{S}$, a graph $G$ is $H$-subgraph-free if and only if it is $H$-minor-free.
\end{lemma}

\begin{theorem}\label{thm:forbidden subgraph}
For every graph $H$, the following conditions are equivalent.
\begin{enumerate}
\item The class of $H$-subgraph-free graphs is $(\tw,\omega)$-bounded.\label[condition]{condition1.1}
\item The class of $H$-subgraph-free graphs has bounded treewidth.\label[condition]{condition1.2}
\item $H \in \mathcal{S}$.\label[condition]{condition1.3}
\end{enumerate}
\end{theorem}

\begin{proof}
Equivalence between~\cref{condition1.1,condition1.2} follows from~\cref{bounded treewidth}.

Suppose now that the class of $H$-subgraph-free graphs has bounded treewidth.
Since the class of elementary walls has unbounded treewidth, $H$ must be a subgraph of some elementary wall.
This implies that $H$ is subcubic.
Suppose next that $H$ contains a connected component with two vertices $u$ and $v$ of degree $3$, and let $\ell$ be the distance between $u$ and $v$.
Then the class of $\ell$-subdivided walls is a subclass of the class of $H$-subgraph-free graphs, and the class of $H$-subgraph-free graphs has unbounded treewidth, a contradiction.
Thus, every connected component of $H$ has at most one vertex of degree~$3$.
Using a similar reasoning, we can conclude that $H$ is acyclic, and thus $H\in \mathcal{S}$.

Finally, suppose that $H \in \mathcal{S}$. Then following \cref{lemma:kS_ell as minor <=> kS_ell as subgraph} every $H$-subgraph-free graph is also $H$-minor-free. Hence, by \cref{planar minor implies bounded treewidth}, the class of $H$-subgraph-free graphs has bounded treewidth.
\end{proof}

A similar approach can be used to prove \cref{thm:forbidden topological minor,thm:H planar is cool}.
We will need the following result.

\begin{lemma}[see, e.g., Diestel \cite{MR3644391}]\label{wall tm or m is equivalent}
A subcubic graph $H$ is a minor of a graph $G$ if and only if $H$ is a topological minor of $G$.
\end{lemma}

\begin{theorem}\label{thm:forbidden topological minor}
For every graph $H$, the following conditions are equivalent.
\begin{enumerate}
\item The class of $H$-topological-minor-free graphs is $(\tw,\omega)$-bounded.\label[condition]{condition2.1}
\item The class of $H$-topological-minor-free graphs has bounded treewidth.\label[condition]{condition2.2}
\item $H$ is subcubic and planar.\label[condition]{condition2.3}
\end{enumerate}
\end{theorem}

\begin{sloppypar}
\begin{proof}
Since every $H$-topological-minor-free graph is also $H$-subgraph-free, \cref{bounded treewidth} implies equivalence between \cref{condition1.1,condition1.2}.

Suppose that the class of \hbox{$H$-topological-minor-free} graphs has bounded treewidth.
Since the class of elementary walls has unbounded treewidth, $H$ is a topological minor of some elementary wall. Thus, since every elementary wall is both subcubic and planar, $H$ must also be subcubic and planar.

Finally, suppose that $H$ is subcubic and planar. Since $H$ is subcubic, by \cref{wall tm or m is equivalent} we obtain that every $H$-topological-minor-free graphs is also $H$-minor-free.
Since $H$ is planar, by \cref{planar minor implies bounded treewidth}, the class of \hbox{$H$-topological-minor-free} graphs has bounded treewidth.
\end{proof}
\end{sloppypar}

\begin{theorem}\label{thm:H planar is cool}
For every graph $H$, the following conditions are equivalent.
\begin{enumerate}
\item The class of $H$-minor-free graphs is $(\tw,\omega)$-bounded.\label[condition]{condition3.1}
\item The class of $H$-minor-free graphs has bounded treewidth.\label[condition]{condition3.2}
\item $H$ is planar.\label[condition]{condition3.3}
\end{enumerate}
\end{theorem}

\begin{proof}
Since every $H$-minor-free graph is also $H$-subgraph-free, we can again invoke \cref{bounded treewidth} to infer that \cref{condition1.1,condition1.2} are equivalent.

Suppose that the class of $H$-minor-free graphs has bounded treewidth.
Since the class of elementary walls has unbounded treewidth, $H$ is a  minor of some elementary wall. Thus, $H$ is planar.

Finally, suppose that $H$ is planar. Then \cref{planar minor implies bounded treewidth} implies that the class of $H$-minor-free graphs has bounded treewidth.
\end{proof}

\section{Algorithmic implications of \texorpdfstring{$(\tw,\omega)$}{(tw,omega)}-boundedness}\label{sec:algo}

As explained in the introduction, the $(\tw,\omega)$-bounded classes having a computable binding function $f$ possess some algorithmically useful properties for variants of the clique and coloring problems.
All the $(\tw,\omega)$-bounded graph classes identified in this work have a computable binding function.
The $(\tw,\omega)$-boundedness of graph classes discussed in~\cref{sec:isitm,sec:im} is derived using either the structure of graphs in the resulting class (\cref{thm:forbidden induced subgraph,lem:block-cactus-graphs}), Ramsey's theorem (\cref{thm:forbidden induced subgraph,K2n-im-free graphs have bdd tw}), or graph minors theory (\cref{cor:K_5-,cor:W4}). 
In the former two cases, there exist binding functions that are explicit polynomials.
In the case of applications of graph minors theory, the key result to deriving those bounds is \cref{thm:FH}, the proof of which relies on results of \citeauthor{MR2802883}~\cite{MR2802883}. As explained in~\cite[Section 9]{MR3936170}, recent developments in the area of graph minors imply that these bounds are computable, too.
For the $(\tw,\omega)$-bounded graph classes discussed in~\cref{sec:stmm}, a result of Chuzhoy and Tan applies stating that if $G$ excludes a planar graph $H$ as a minor, then the treewidth of $G$ is $\mathcal{O}(|V(H)|^{9}\,\text{poly}\log |V(H)|)$~\cite{MR4155282}.
An explicit upper bound $\tw(G)\le 2^{15|V(H)|+8|V(H)|\log(|V(H)|)}$ was also shown in~\cite{MR3315599}.
For later use, we record this observation in the form of a theorem.

Let us denote by $\Sigma$ the family of $(tw,\omega)$-bounded graph classes excluding a fixed graph $H$ as a subgraph, a topological minor, or a minor (cf.\ the middle column of \cref{table-results}).
Similarly, we denote by $\Sigma_i$ the family of $(tw,\omega)$-bounded graph classes excluding a fixed graph $H$ as an induced subgraph, an induced topological minor, or an induced minor (cf.\ the right column of \cref{table-results}).

\begin{theorem}\label{thm:computable-binding-function}
Each graph class $\mathcal{G}\in \Sigma \cup \Sigma_i$ has a computable $(\tw,\omega)$-binding function,
which is constant if $\mathcal{G}\in \Sigma$.
\end{theorem}

As already mentioned in the introduction, Chaplick and Zeman showed in~\cite{DBLP:journals/endm/ChaplickZ17} that for every positive integer $k$, there exists a linear-time algorithm for the \textsc{$k$-Clique} and \textsc{List $k$-Coloring} problems in any $(\tw,\omega)$-bounded graph class having a computable binding function.
Combining this result with~\cref{thm:computable-binding-function} yields the following.

\begin{corollary}\label{algorithmic-k-clique-k-col}
For every positive integer $k$, there exists a linear-time algorithm for the \textsc{$k$-Clique} and \textsc{List $k$-Coloring} problems in each graph class $\mathcal{G}\in \Sigma \cup \Sigma_i$. In particular, this holds when $\mathcal{G}$ is either the class of block-cactus graphs or the class
of $H$-induced-minor-free graphs for some
$H\in \{W_4,K_5^-\}\cup\{K_{2,q}\mid q\in \mathbb{N}\}$.
\end{corollary}

\citeauthor{MR3992956} gave in~\cite{MR3992956} a polynomial-time algorithm for the \textsc{Maximum Weight Clique} problem in a class of graphs generalizing the class of $1$-perfectly orientable graphs. They asked about the complexity of the \textsc{Maximum Independent Set} and \hbox{\textsc{$k$-Coloring}} problems in the  class of $1$-perfectly orientable graphs. Since every $1$-perfectly orientable graph is $K_{2,3}$-induced-minor-free, \cref{algorithmic-k-clique-k-col} answers, in a much greater generality, the question by \citeauthor{MR3992956} on the complexity of $k$-coloring $1$-perfectly orientable graphs.

Given the useful algorithmic properties of $(\tw,\omega)$-bounded graph classes, it would be good to have a polynomial-time recognition algorithm for graphs in any such class.
Graphs $G$ excluding a fixed graph $H$ either as a minor or as a topological minor can be recognized in time $\mathcal{O}(|V(G)|^3)$~\cite{MR1309358,MR2931998}. Clearly, graphs excluding a fixed graph $H$ as a subgraph or an induced subgraph can be recognized in polynomial time, simply by checking all the $\mathcal{O}(|V(G)|^{|V(H)|})$ subsets of vertices of size $|V(H)|$.
The situation is less clear for graphs excluding a single induced minor or induced topological minor, as there exist graphs $H$ such that it is \textsf{co-NP}-complete to recognize $H$-induced-minor-free graphs or $H$-induced-topological-minor-free graphs~\cite{MR1308575,MR2551944}.
Nevertheless, for all the $(\tw,\omega)$-bounded classes of $H$-induced-topological-minor-free graphs (characterized by \cref{thm:itm-free}), the recognition problem is easily observed to be polynomial-time solvable due to the special structure of these graph classes.
They are the classes of chordal graphs (if $H \cong C_4$), of block-cactus graphs (if $H \cong K_4^-$), of $P_3$-free graphs (if $H \cong P_3$), of acyclic graphs (if $H \cong C_3$), of edgeless graphs (if $H \cong P_2$), and of graphs of bounded independence number (if $H$ is edgeless).
Furthermore, it can be seen that a graph $G$ has $K_{1,q}$  as an induced minor if and only if $G$ has an independent set $S$ of size $q$ such that for some connected component $C$ of $G-S$, every vertex in $S$ has a neighbor in $C$. This implies that the recognition problem for the class of $K_{1,q}$-induced-minor-free graphs is polynomial-time solvable.
Among the $(\tw,\omega)$-bounded classes of $H$-induced-minor-free graphs (cf.~\cref{thm:necessary forbidden induced minors}), the cases when $H\in \{C_4, K_4^-, C_3, P_3, P_2\}$ or $H$ is edgeless are the same as above and hence recognizable in polynomial time.
The complexity of recognition remains open for $H \cong W_4$, $H \cong K_5^-$, or $H \cong K_{2,q}$ or $H \cong K_{2,q}^+$ for some $q\ge 3$.
As we explain next, this is not necessarily a problem.

As shown by Chaplick and Zeman, for every
$(\tw,\omega)$-bounded class $\mathcal{G}$ with a computable binding function and for every fixed $k$, \textsc{List $k$-Coloring} is solvable in linear time
for graphs in $\mathcal{G}$~\cite{DBLP:journals/endm/ChaplickZ17}.
If we are satisfied with polynomial running time, we can extend their approach to obtain an algorithm for \textsc{List $k$-Coloring} that is \emph{robust} in the sense of Raghavan and Spinrad~\cite{MR2006100}: it either solves the problem or determines that the input graph is not in $\mathcal{G}$.

\begin{theorem}\label{thm:motivation}
Let $\mathcal{G}$ be a $(\tw,\omega)$-bounded graph class having a computable $(\tw,\omega)$-binding function $f$. Then, for every positive integer $k$ there exists a robust polynomial-time algorithm for the \textsc{List $k$-coloring} problem on graphs in $\mathcal{G}$.
\end{theorem}

\begin{proof}
The algorithm is as follows. First, we test whether the input graph $G$ contains a clique of size $k+1$ in time $\mathcal{O}(|V(G)|^{k+1})$. If it does, then $G$ is not \hbox{$k$-colorable}. Suppose it does not. Then $\omega(G)\le k$. In particular, this means that if $G\in \mathcal{G}$, then
$\tw(G)\le f(\omega(G))\le c_k$, where $c_k = \max\{f(1),\ldots, f(k)\}$.
Using the linear-time algorithm of Bodlaender~\cite{MR1417901}, we test whether $\tw(G)\le c_k$.
If $\tw(G)>c_k$, then $G\not\in \mathcal{G}$ and the algorithm returns the message ``$G\not\in \mathcal{G}$.''
If $\tw(G)\le c_k$, then the algorithm of Bodlaender actually computes a tree decomposition of $G$ of width at most $c_k$. Using this tree decomposition, we can now invoke a result of Jansen and Scheffler~\cite{MR1451958} to test in linear time whether $G$ is $k$-colorable with respect to the given lists of available colors for each vertex.

The correctness of the algorithm is obvious.
The running time of the algorithms by Bodlaender and by Jansen and Scheffler is $\mathcal{O}(g(c_k)(|V(G)|+|E(G)|))$ and $\mathcal{O}(h(c_k)(|V(G)|+|E(G)|))$, respectively, for some functions $g$ and $h$ depending only on $c_k$ (and thus only on $k$). Thus, the total running time of the algorithm is $\mathcal{O}(|V(G)|^{k+1} + (g(c_k)+h(c_k))(|V(G)|+|E(G)|))$, which is polynomial in the input size for every fixed value of $k$.
\end{proof}

We next discuss some possible implications of $(\tw,\omega)$-boundedness for improved approximations for the \textsc{Maximum Clique} problem: given a graph $G$, find a maximum clique in $G$.
For general graphs, this problem is notoriously difficult to approximate: for every $\varepsilon>0$, there is no polynomial-time algorithm for approximating the maximum clique in an $n$-vertex graph to within a factor of $n^{1-\varepsilon}$ unless $\P = \NP$ \cite{MR2403018}.
An approximation algorithm for an optimization problem is typically required to compute a feasible solution to the problem.
As we explain next, for $(\tw,\omega)$-bounded graph classes with a polynomial binding function, known approximation algorithms for treewidth (see, e.g.,~\cite{MR2411037}) lead to an improved approximability bound, provided that we allow the algorithm to output only a number approximating the value of the optimal solution and not the approximate solution itself.
We denote by \textsf{opt} the optimal solution value of the maximum clique problem on the input graph $G$, that is, $\omega(G)$.

\begin{theorem}\label{thm:motivation-2}
Let $\mathcal{G}$ be a graph class having a computable polynomial $(\tw,\omega)$-binding function \hbox{$f(k) = \mathcal{O}(k^c)$} for some constant $c$. Then, for all $\varepsilon>0$ the clique number can be approximated for graphs in $\mathcal{G}$ in polynomial time to within a factor of
$\textsf{opt}^{1-1/(c+\varepsilon)}$.
\end{theorem}

\begin{proof}
Fix an $\varepsilon>0$ and let $G\in \mathcal{G}$. Using the algorithm of \citeauthor{MR2411037}~\cite{MR2411037}, we can compute in polynomial time a tree decomposition of $G$ of width
\[
    t = \mathcal{O}(\tw(G)\sqrt{\log \tw(G)})\,.
\]
Since $G\in \mathcal{G}$, the assumption on $\mathcal{G}$ implies that $\tw(G)= \mathcal{O}(\omega(G)^c)$.
Consequently,
\[
    t+1 = \mathcal{O}(\omega(G)^c\sqrt{\log (\omega(G)^c)}) =
    \mathcal{O}(\omega(G)^c(\log (\omega(G)))^{1/2})\,.
\]
This implies that $t+1\le \omega(G)^{c+\varepsilon}$ as soon as $\omega(G)\ge k_0$ for a suitable constant $k_0$ depending only on $\varepsilon$, $c$, and the constants hidden in the $\mathcal{O}$ notation of the approximation ratio of the algorithm of \citeauthor{MR2411037} and of the binding function.
Note that the assumption $\omega(G)\ge k_0$ is without loss of generality since otherwise we can compute $\omega(G)$ in polynomial time.
We thus have $\omega(G) \ge (t+1)^{1/(c+\varepsilon)}$, and this lower bound can be computed in polynomial time.
Since $\omega(G)-1\le \tw(G)\le t$, we have $(t+1)^{1/(c+\varepsilon)}\ge \omega(G)^{1/(c+\varepsilon)}$. This means that the lower bound
$(t+1)^{1/(c+\varepsilon)}$ approximates the value of the clique number $\omega(G)$ to within a factor of $\omega(G)^{1-1/(c+\varepsilon)}$, as claimed.
\end{proof}

Note that unless $\P = \NP$, the result of~\cref{thm:motivation-2} cannot be improved by means of using a polynomial-time algorithm for computing the treewidth in  $(\tw,\omega)$-bounded graph classes (which would allow taking $\varepsilon = 0$), since there exist graph classes with a linear $(\tw,\omega)$-binding function in which the treewidth is \NP-hard to compute. In fact, the original \NP-hardness proof for computing the treewidth due to \citeauthor{MR881187}~\cite{MR881187} produces co-bipartite graphs, and since the vertex set of every co-bipartite graph $G$ can be covered by two cliques, we have $\tw(G)\le |V(G)|-1 \le 2\omega(G)-1$.

\begin{corollary}\label{cor:approx-omega-in-linearly-tw-omega-bounded-classes}
Let $\mathcal{G}$ be a graph class having a computable linear $(\tw,\omega)$-binding function. Then, for all $\varepsilon>0$ the clique number can be approximated for graphs in $\mathcal{G}$ in polynomial time to within a factor of $\textsf{opt}^{\varepsilon}$.
\end{corollary}

Note that the result of \cref{cor:approx-omega-in-linearly-tw-omega-bounded-classes} cannot be improved to a polynomial-time approximation scheme for the maximum clique problem, unless $\P = \NP$, as there exist graph classes with a linear $(\tw,\omega)$-binding function in which the clique number is \APX-hard to compute; see~\cite{DBLP:journals/endm/ChaplickZ17}.

For exponential binding functions, the same approach leads to an improvement over the trivial $\textsf{opt}$-approximation to the maximum clique (return any vertex), as follows.

\begin{theorem}\label{thm:motivation-2'}
Let $\mathcal{G}$ be a graph class having a computable exponential $(\tw,\omega)$-binding function $f$, say $f(k) = \mathcal{O}(c^k)$ for some constant $c>1$. Then, the clique number can be approximated for graphs in $\mathcal{G}$ in polynomial time to within a factor of $\mathcal{O}(\textsf{opt}/\log \textsf{opt})$.
\end{theorem}

\begin{proof}
Let $G\in \mathcal{G}$. Using the algorithm of \citeauthor{MR2411037}~\cite{MR2411037}, we can compute in polynomial time a tree decomposition of $G$ of  width $t = \mathcal{O}(\tw(G)\sqrt{\log \tw(G)})$. Since $G\in \mathcal{G}$, we have $\tw(G)= \mathcal{O}(c^{\omega(G)})$.
Consequently,
$$t =
\mathcal{O}\left(c^{\omega(G)}\sqrt{\log (c^{\omega(G)})}\right) =
\mathcal{O}(c^{\omega(G)}\sqrt{\omega(G)})\,.$$
It follows that in polynomial time we can compute a lower bound on the clique number $\omega(G)$ of the form
$\Omega(\log t)$. Since $\omega(G)-1\le \tw(G)\le t$, this lower bound is of the order
$\Omega(\log \omega(G))$. This means that it approximates the value of the clique number $\omega(G)$ to within a factor of
$\mathcal{O}(\omega(G)/\log \omega(G))$.
\end{proof}

\begin{remark}
The proofs of~\cref{thm:motivation,thm:motivation-2,thm:motivation-2'} remain valid as soon as the inequality $\tw(G)\le f(\omega(G))$ holds for all graphs $G\in \mathcal{G}$, and not necessarily for all induced subgraphs of graphs in $\mathcal{G}$.
Thus, all these results, along with \cref{cor:approx-omega-in-linearly-tw-omega-bounded-classes}, also hold for such more general graph classes, which need not be closed under induced subgraphs.
\end{remark}

\section{Remarks on \texorpdfstring{$(\tw,\omega)$}{(tw,omega)}-binding functions of graphs of bounded independence number}\label{sec:lower-bounds}

We now derive bounds on the degrees of the
$(\tw,\omega)$-binding polynomials for classes of graphs of bounded independence number.
First, note that \cref{thm:forbidden induced subgraph} implies that for all positive integers $q$, the class of graphs with independence number less than $q$ is $(\tw,\omega)$-bounded with a polynomial binding function $f(k) = \mathcal{O}(k^{q-1})$.
As observed by Trotignon and Pham~\cite{MR3789676} (see also~\cite{esperet2017graph,MR4174126}), classes of graphs of bounded independence number form a family of polynomially $\chi$-bounded graph classes that require $\chi$-binding polynomials of arbitrarily large degrees.
Note that for every $(\tw,\omega)$-binding function $f(k)$ for some graph class, the function $f(k)+1$ is a $\chi$-binding function for the same class.
Therefore, classes of graphs of bounded independence number also require $(\tw,\omega)$-binding polynomials of arbitrarily large degrees.
For the sake of completeness, we include the proof with slightly better bounds than the ones that follow from~\cite{MR3789676}.

The proof of the lower bound is based on the following lower bounds on the Ramsey numbers proved by Spencer~\cite{MR491337}.

\begin{theorem}\label{thm:Ramsey-lower-bound}
For every integer $q\ge 3$ there exists a constant $c_q>0$ and a positive integer $k_q$ such that for all $k\ge k_q$ we have
$$R(k,q) \ge c_q\left(\frac{k}{\log k}\right)^{(q+1)/2}\,.$$
\end{theorem}

\begin{theorem}\label{thm:tw-omega}
For every integer $q\ge 3$, let $\mathcal{G}_q$ denote the class of graphs $G$ with $\alpha(G)<q$ (that is, the class of $qK_1$-free graphs).
Then, $\mathcal{G}_q$ is a $(\tw,\omega)$-bounded graph class such that
\begin{itemize}
	\item $\mathcal{G}_q$ has a $(\tw,\omega)$-binding function that is a polynomial of degree $q-1$;
	\item for every $c>0$ and every $\epsilon>0$, the function
	$$f(k) = ck^{(q+1)/2-\epsilon}$$ is not a
	$(\tw,\omega)$-binding function for $\mathcal{G}_q$.
\end{itemize}
\end{theorem}

\begin{proof}
Fix an integer $q\ge 3$, and let $c_q$ and $k_q$ be the corresponding constants given by \cref{thm:Ramsey-lower-bound}.
The class $\mathcal{G}_q$ is $(\tw,\omega)$-bounded, since if $k$ is a positive integer and $G\in \mathcal{G}_q$ has clique number $k$, then $\tw(G)\le |V(G)|-1 < R(k+1,q)-1 = \mathcal{O}(k^{q-1})$.

Suppose for a contradiction that for some $c>0$ and
some $\epsilon>0$, we have
\begin{equation}\label{eq1}
\tw(G)\le c(\omega(G))^{(q+1)/2-\epsilon}
\end{equation}
for all $G\in \mathcal{G}_q$.

For each integer $k\ge k_q$, let $n_k$ be the largest integer such that
\[
	n_k<c_q\left(\frac{k}{\log k}\right)^{(q+1)/2}\,.
\]
Then \cref{thm:Ramsey-lower-bound} implies that $n_k<R(q,k)$.
By the definition of the Ramsey number $R(q,k)$, there exists a graph $G_k$ with exactly $n_k$ vertices that has no independent set of size $q$ and no clique of size $k$.
Since $\alpha(G_k)<q$, we have $G_k\in \mathcal{G}_q$.
Furthermore, since $|V(G_k)|\le \alpha(G_k)\cdot\chi(G_k)$
and $\chi(G_k)\le \tw(G_k)+1$ (see \cref{thm:chi}),
\cref{eq1} implies
\begin{eqnarray*}
n_k &=& |V(G_k)|\\&\le& \alpha(G_k)\cdot\chi(G_k)\\&<& q\cdot(\tw(G_k)+1)\\&\le &q\cdot(c(\omega(G_k))^{(q+1)/2-\epsilon}+1)\\
&<& q(ck^{(q+1)/2-\epsilon}+1)\,.
\end{eqnarray*}
By the definition of $n_k$, we have
$n_k\ge c_q\left(\frac{k}{\log k}\right)^{(q+1)/2}-1$.
Comparing this lower bound on $n_k$ with the above upper bound for $n_k$, we derive
\begin{equation}\label{eq3}
c_q\left(\frac{k}{\log k}\right)^{(q+1)/2}-1<q(ck^{(q+1)/2-\epsilon}+1)\,.
\end{equation}
Since $q$, $c_q$, $c$, and $\epsilon$ are fixed constants independent of $k$, \cref{eq3}
will be violated for all sufficiently large $k$, a contradiction.
\end{proof}

\Cref{thm:tw-omega} implies the following.

\begin{corollary}
There exists no polynomial that is a $(\tw,\omega)$-binding function for all polynomially  $(\tw,\omega)$-bounded graph classes.
\end{corollary}

\section{Remarks on the \textsc{Maximum Weight Independent Set} problem}\label{section MWIS}

The \textsc{Maximum Weight Independent Set} (MWIS) problem takes as input a graph $G$ and a weight function $w:V(G)\to \mathbb{Q}_+$, and the task is to find an independent set $I$ in $G$ of maximum possible weight $w(I)$, where $w(I) = \sum_{x\in I}w(x)$.
The computational complexity of the MWIS problem for the $(\tw,\omega)$-bounded graph classes listed in \cref{table-results} can be summarized as follows.
\begin{enumerate}
\item As shown in \cref{sec:stmm}, $(\tw,\omega)$-boundedness is equivalent to bounded treewidth when $H$ is forbidden as a subgraph, topological minor, or minor. Furthermore, the upper bound on the treewidth is a computable constant (cf.~\cref{thm:computable-binding-function}).
Thus, in all these cases, the MWIS problem can be solved in linear time.
Indeed, using the linear-time algorithm of Bodlaender~\cite{MR1417901} we can compute a tree decomposition of the input graph $G$ of constant width and then use it to compute a maximum weight independent set in linear time, following, e.g., the approach of \citeauthor{MR1105479}~\cite{MR1105479}.

\item When $H$ is forbidden as an induced subgraph, following \cref{thm:forbidden induced subgraph}, the resulting graph class is  $(\tw,\omega)$-bounded if and only if $H$ is either an induced subgraph of $P_3$ or an edgeless graph. When $H\subseteq_{is} P_3$, every $H$-free graph is a disjoint union of complete graphs; thus computing a maximum weight independent set for any $H$-free graph is equivalent to finding a vertex of maximum weight in each clique, which can be done in linear time. If $H$ is edgeless, then every $H$-free graph contains only independent sets of size at most $|V(H)|-1$. Thus, we can enumerate all independent sets of $G$ and compute one of maximum weight in time $\mathcal{O}(|V(G)|^{|V(H)|-1})$.

\item When $H$ is forbidden as an induced topological minor, following \cref{thm:itm-free}, the resulting graph class is  $(\tw,\omega)$-bounded if and only if $H$ is either an induced subgraph of $K_4^-$, of $C_4$, or an edgeless graph.

When $H \subseteq_{is} C_4$, the class of $H$-induced-topological-minor free graphs is a subclass of the class of chordal graphs for which a linear-time algorithm for the MWIS problem is known~\cite{MR0392683}.

When $H$ is edgeless, every $H$-induced-topological-minor-free graph is also $H$-free, and hence a maximum weight independent set in $G$ can be computed in time $\mathcal{O}(|V(G)|^{|V(H)|-1})$, as above.

If $H\subseteq_{is} K_4^-$, then by \cref{diamond-itm-free graphs are block-cactus-graphs}, every $H$-induced-topological-minor-free graph is a block-cactus graph, and a polynomial-time solvability for the MWIS problem follows from the fact that the clique-width of every block-cactus graphs is at most $6$~\cite{MR2536473} and the algorithm given by a metatheorem of \citeauthor{MR1739644}~\cite{MR1739644}.

\item Finally, when $H$ is forbidden as an induced minor, following \cref{thm:necessary forbidden induced minors}, the resulting graph class is  $(\tw,\omega)$-bounded if and only if $H$ is either an induced subgraph of $W_4$, of $K_5^-$, of
$K_{2,q}$, or of $K_{2,q}^+$ for some $q\in \mathbb{N}$.
The computational complexity of the MWIS problem is open for the class of $H$-induced-minor-free graphs whenever $H$ is isomorphic to
$W_4$, to $K_5^-$, to $K_{2,q}$, or to $K_{2,q}^+$ for some $q\ge 3$. However, as we show next, the problem is solvable in polynomial time in the class of $K_{1,q}$\=/induced-minor-free graphs for every positive integer $q$.
\end{enumerate}

\begin{theorem}\label{thm:K_{1,q}}
For every positive integer $q$, there is a polynomial-time algorithm for the
\textsc{Maximum Weight Independent Set} problem in the class of $K_{1,q}$\=/induced-minor-free graphs.
\end{theorem}

\begin{proof}
Fix a positive integer $q$, and let $\G$ be the class of $K_{1,q}$-induced-minor-free graphs.
Consider a graph $G\in \G$ equipped with a weight function $w:V(G)\to \mathbb{Q}_+$.
We may assume without loss of generality that $G$ is connected, since otherwise we can solve the
problem separately for each connected component and combine the solutions.

Fix a vertex $v\in V(G)$, and run breadth-first search from $v$ to compute the distance levels from $v$, that is, the sets $N_i$ for all $i\ge 0$, where $N_i$ is the set of vertices in $G$ that are at distance $i$ from $v$. In particular, $N_0 = \{v\}$. We denote by $\ell$ the eccentricity of $v$, that is, the maximum $i\ge 0$ such that $N_i\neq \emptyset$.

The algorithm will be based on the observation that, for each $i\in \{1,\ldots, \ell\}$, the subgraph of $G$ induced by $N_i$ has independence number less than $q$. Indeed, if $G$ contains an independent set $I$ such that $|I|\ge q$ and $I\subseteq N_i$ for some $i\in \{1,\ldots, \ell\}$, then the set $\cup_{j<i}N_j$ (which induces a connected subgraph of $G$) along with the singletons $\{x\}$,  $x\in I$, form an induced minor model of $K_{1,q}$ in $G$, a contradiction.

The algorithm uses a dynamic programming approach. It processes the distance levels from $v$ one at a time,
trying all possibilities of which vertices from the current distance level will appear in an independent set that is a partial solution for the subgraph of $G$ induced by the distance levels considered so far.
Since we cannot include more than $q-1$ vertices from any distance level $N_i$ for $i\ge 1$, we only need a polynomial number of guesses for each distance level.
Formally, the algorithm computes the following values:
for each $i\in \{0,1,\ldots, \ell\}$ and each independent sets $S$ such that $S\subseteq N_{i}$,
the value $\alpha[i,S]$, defined as the maximum weight of an independent set $I$ in $G$ such that $I\subseteq \cup_{j\le i}N_j$ and $I\cap N_i = S$.
Knowing all these values $\alpha[i,S]$, we can then compute the maximum weight of an independent set in $G$ as the maximum value of $\alpha[\ell,S]$ over all independent sets $S\subseteq N_{\ell}$.
Indeed, if $I^*$ is an arbitrary maximum weight independent set in $G$ and $S^* = I^*\cap N_{\ell}$, then $\alpha[\ell,S^*] \ge w(I^*)$,
while, clearly, for any independent set $S\subseteq N_{\ell}$, we have $\alpha[\ell,S]\le w(I^*)$.

Let us describe the recurrence relation for the values of $\alpha[i,S]$ forming the dynamic programming table.
For all $i\in \{0,1,\ldots, \ell\}$ and all independent sets $S\subseteq N_i$, we have
\[
\alpha[i,S] = \left\{
\begin{array}{ll}
w(S)&\text{if } i = 0\\
w(S)+\max \{\alpha[i-1,T] \mid T\subseteq N_{i-1} \textrm{ and }& \\
\hspace{4.2cm}\textrm{$S\cup T$ is independent in $G$} \} &\text{otherwise.}
\end{array}
\right.
\]
The validity of the above recurrence relation can be easily verified. Thus, the algorithm is correct.

It remains to analyze the time complexity of the algorithm. Observe that there are at most $|V(G)|$ distance levels.
For each distance level $N_i$, $i \in \{0,\dots,\ell\}$, the algorithm enumerates all independent sets
$S\subseteq N_i$. For the distance level $N_0$, we only have two choices for $S$ (either $S = \emptyset$ or $S = \{v\}$).
For distance levels $N_i$, $i \in \{1,\dots,\ell\}$, there are $\sum_{0 \leq k \leq q-1} \binom{|N_{i}|}{k} = \mathcal{O}(|N_i|^{q-1}) = \mathcal{O}(|V(G)|^{q-1})$ choices for the set $S$.
Assuming that the graph $G$ is given by the adjacency matrix representation, we can check in constant time
$\mathcal{O}(q^2)$ whether a set of at most $q-1$ vertices is independent.
Thus, all these independent sets can be enumerated in total time $\mathcal{O}(|V(G)|^{q})$.
For each distance level $N_i$ with $i\ge 1$, we iterate over all $\mathcal{O}(|V(G)|^{q-1})$ independent sets
$S\subseteq N_{i}$ and, for each such set $S$, over all independent sets $T\subseteq N_{i-1}$.
For such a pair $(S,T)$ we can check in (constant) time $\mathcal{O}(q^2)$ whether $S\cup T$ is independent, and if so, we record the value of $\alpha[i-1,T]$. Thus, for a fixed independent set $S\subseteq N_{i}$, we can compute the value of $\alpha[i,S]$ in constant time if $i = 1$ and in time $\mathcal{O}(|V(G)|^{q-1})$ if $i\ge 2$.
It follows that all the values of $\alpha[i,S]$ can be computed in time
$\mathcal{O}(|V(G)|^{2q-1})$.
Finally, the maximum weight of an independent set in $G$ can be computed in time
$\mathcal{O}(|V(G)|^{q-1})$ by computing the maximum value of $\alpha[\ell,S]$ over all independent sets $S \subseteq N_{\ell}$.
The overall worst-case time complexity of the algorithm is $\mathcal{O}(|V(G)|^{2q-1})$.
An optimal solution can also be computed if every time we compute the value of $\alpha[i,S]$, we also compute an independent set $I$ achieving the maximum in the definition of $\alpha[i,S]$.
\end{proof}

Let us note that the result of \cref{thm:K_{1,q}} contrasts with the fact that if $K_{1,q}$ is only excluded as induced subgraph, then the MWIS problem
\begin{itemize}
\item can be solved optimally in polynomial time for $q\le 3$ (see~\cite{MR579076,MR1845110}),
\item can be approximated in polynomial time to within a factor of $q/2$ for all $q$ (see~\cite{MR1810314}),
\item remains \APX-hard for all $q\ge 4$ already for unit weight functions (see~\cite{MR1756204}).
\end{itemize}

\section{Discussion}\label{sec:discussion}

We obtained a first set of results aimed towards classifying $(\tw,\omega)$-bounded graph classes by considering six well-known graph containment relations and, for each of them, characterizing the graphs $H$ for which the class of graphs excluding $H$ is $(\tw,\omega)$-bounded. We conclude the paper by mentioning some open questions and research directions for further investigations of $(\tw,\omega)$-bounded graph classes.

\subsection{Larger sets of forbidden structures}

A natural question motivated by the results of this paper is

\begin{question}\label{question1}
Which graph classes defined by larger finite sets of forbidden structures (with respect to various graph containment relations) are $(\tw,\omega)$-bounded?
\end{question}

For the subgraph, topological minor, and minor relations, the answers to this question are similar to the results for a single forbidden structure developed in~\cref{sec:stmm}, as we explain next.
\begin{itemize}
\item Robertson and Seymour showed that every minor-closed graph class that does not equal the class of all graphs can be characterized by a finite set of forbidden minors (see, e.g.,~\cite{MR3644391}). The same approach as that used to prove \cref{thm:H planar is cool} shows that a proper minor-closed graph class is $(\tw,\omega)$-bounded if and only if it excludes some planar graph.
\item The proof of \cref{thm:forbidden topological minor} can also be easily  generalized to show that a graph class defined by excluding finitely many topological minors is $(\tw,\omega)$-bounded if and only if it excludes at least one graph that is subcubic and planar.
\item The case of subgraphs is only slightly more complicated, but again the proof of \cref{thm:forbidden subgraph} can be adapted to show that a graph class defined by excluding finitely many subgraphs is $(\tw,\omega)$-bounded if and only if it excludes at least one graph from the class $\mathcal{S}$.
\end{itemize}

In all the above cases, the resulting graph class is $(\tw,\omega)$-bounded if and only if it has bounded treewidth.

For the induced subgraph relation, the answer to~\cref{question1} is a direct consequence of the following recent characterization of graph classes of bounded treewidth defined by finitely many forbidden induced subgraphs.

\begin{theorem}[Lozin and Razgon~\cite{lozin2021treewidth}]\label{LozinRazgon}
For any graphs $H_1,\ldots, H_p$, the class of $\{H_1,\ldots, H_p\}$-free graphs has bounded treewidth if and only if the set $\{H_1,\ldots, H_p\}$ contains a complete graph, a complete bipartite graph, a graph from $\mathcal{S}$, and the line graph of a graph from $\mathcal{S}$.
\end{theorem}

\begin{corollary}
For any graphs $H_1,\ldots, H_p$, the following conditions are equivalent.
\begin{enumerate}
\item The class of $\{H_1,\ldots, H_p\}$-free graphs is $(\tw,\omega)$-bounded.\label[condition]{condition 1}
\item The class of $\{K_4,H_1,\ldots, H_p\}$-free graphs has bounded treewidth.\label[condition]{condition 2}
\item The set $\{H_1,\ldots, H_p\}$ contains a complete bipartite graph, a graph from $\mathcal{S}$, and the line graph of a graph from $\mathcal{S}$.\label[condition]{condition 3}
\end{enumerate}
\end{corollary}

\begin{proof}
If the class of $\{H_1,\ldots, H_p\}$-free graphs has a $(\tw,\omega)$-binding function~$f$, then the treewidth of any $\{K_{4},H_1,\ldots, H_p\}$-free graph is at most $f(3)$. Thus, \cref{condition 1} implies \cref{condition 2}.
By~\cref{LozinRazgon}, \cref{condition 2} implies \cref{condition 3}.
Finally, if the set $\{H_1,\ldots, H_p\}$ contains a complete bipartite graph, a graph from $\mathcal{S}$, and the line graph of a graph from $\mathcal{S}$, then by~\cref{LozinRazgon} for every positive integer $k$ there exists a constant $f(k)$ such that every $\{K_{k+1},H_1,\ldots, H_p\}$-free graph has treewidth at most $f(k)$.
Thus, the class of $\{H_1,\ldots, H_p\}$-free graphs is $(\tw,\omega)$-bounded and \cref{condition 3} implies \cref{condition 1}.
\end{proof}

For the induced topological minor and induced minor relations, \cref{question1} remains open.

\subsection{Binding functions}

All the $(\tw,\omega)$-bounded graph classes identified in this paper have a polynomial $(\tw,\omega)$-binding function.
The arguments presented in this paper establish this claim for all cases, except possibly for the classes of $H$-induced-minor-free graphs with $H\in \{W_4,K_5^-\}$, where the proof depends on the constant involved in the excluded minor theorem of Robertson and Seymour~\cite{MR1999736}. 
Polynomial $(\tw,\omega)$-binding functions for these two cases will be established in a future publication.
A natural question arises.

\begin{question}\label{question2}
Does every $(\tw,\omega)$-bounded graph class have a polynomial $(\tw,\omega)$-binding function?
\end{question}

A positive answer to this question would answer the analogous question of Esperet (see~\cite{esperet2017graph}) on $\chi$-boundedness for the case of $(\tw,\omega)$-bounded graph classes.

For every positive integer $t$, the class of intersection graphs of connected subgraphs of graphs with treewidth at most $t$ has a linear $(\tw,\omega)$-binding function (see~\cite{MR1090614}). For $t\in \{1,2\}$, the corresponding graph classes contain the well-known graph classes of chordal and circular-arc graphs.
More generally, it is an interesting question
to determine necessary and/or sufficient conditions for the existence of a linear $(\tw,\omega)$-binding function.

\begin{question}\label{question3}
Which graph classes have a linear $(\tw,\omega)$-binding function?
\end{question}

\subsection{Algorithmic implications}

We already discussed in \cref{sec:algo} some algorithmic implications of $(\tw,\omega)$-boundedness for variants of the  clique and coloring problems. In particular,  \cref{thm:motivation-2,cor:approx-omega-in-linearly-tw-omega-bounded-classes} give approximation algorithms for the clique number in graph classes having a polynomial, respectively, linear $(\tw,\omega)$-binding function. 
In order to strengthen these results to obtain improved approximation algorithms for the \textsc{Maximum Clique} problem that would actually compute an approximate solution and not only an approximate value, we would need to know something more about the $(\tw,\omega)$-bounded graph class $\mathcal{G}$.

\begin{question}\label{question4}
Which graph classes $\mathcal{G}$ having a polynomial $(\tw,\omega)$-binding function~$f$ admit a polynomial-time algorithm that, given a tree decomposition of width $t$ for a graph $G\in \mathcal{G}$, computes a clique $C$ in $G$ such that $t = \mathcal{O}(f(|C|))$?
\end{question}

It is a well-known open problem whether treewidth can be approximated within a constant factor (see, e.g.,~\cite{MR2411037,DBLP:journals/jair/WuAPL14}).
Perhaps the following special case could be easier.

\begin{question}\label{question5}
Can the treewidth be approximated within a constant factor on $(\tw,\omega)$-bounded classes?
\end{question}

The question is also open if an additional constraint is imposed on the binding function, for example, that it is polynomial or linear. Note that for graph classes with a linear $(\tw,\omega)$-binding function, a constant factor approximation for treewidth would also imply a constant factor approximation for the clique number.

Finally, it would be interesting to see whether $(\tw,\omega)$-boundedness has any further algorithmic implications, for example, for problems related to independent sets.
The computational complexity of the MWIS problem in $(\tw,\omega)$-bounded graph classes is not yet well understood.
Nonetheless, partial results exist for $(\tw,\omega)$-bounded graph classes forbidding a unique graph $H$ with respect to one of our six graph containment relations, as discussed in \cref{section MWIS}.
To the best of our knowledge, the following question is open.

\begin{question}\label{question6}
Is there a $(\tw,\omega)$-bounded graph class in which the MWIS problem is \NP-hard?
\end{question}

\subsection{Chromatic number versus clique number and other invariant pairs}

A similar study as the one performed in this work could, at least in principle, be attempted for $(\rho,\sigma)$-bounded graph classes for other choices of pairs of invariants $\rho$ and $\sigma$.
Such questions were partially already considered in the literature, for example, for the pair $(\rho,\sigma)=(\delta, \chi)$~\cite{MR2811077} and, most notably, for the pair $(\rho,\sigma)=(\chi, \omega)$, that is, for the case of $\chi$-bounded graph classes (see~\cite{MR4174126}).
Dirac~\cite{MR195757} and Jung~\cite{MR188101} proved
that excluding any complete graph $H$ as a topological minor
leads to a class of graphs with bounded chromatic number, and consequently the same is true if any graph $H$ is excluded as a topological minor or as a minor. Furthermore, it follows from~\cite{MR102081} and~\cite[Exercise 5.2.43]{MR1367739} that the class of $H$-subgraph-free graphs is $\chi$-bounded if and only if $H$ is acyclic (in which case the chromatic number is bounded). For the induced variants of these relations, the situation is less clear.
While the class of $H$-free graphs is not $\chi$-bounded whenever $H$ contains a cycle~\cite{MR102081}, a famous and still open conjecture, proposed independently by Gy\'arf\'as~\cite{MR951359} and Sumner~\cite{MR634555}, states that for every tree $H$, every class of $H$-free graphs is $\chi$-bounded.
Furthermore, Scott proved in~\cite{MR1437291} that for any tree $H$, the class of $H$-induced-topological-minor-free graphs is $\chi$-bounded and conjectured that the same is true for any graph $H$.
While Scott's conjecture has been disproved (see, e.g.,~\cite{MR3484735}) and several special cases were proved, for example, the case when $H$ is cycle~\cite{MR3759907} (settling yet another conjecture of Gy\'arf\'as),
we are still far from a complete characterization of graphs $H$ such that the class of $H$-induced-topological-minor-free graphs is $\chi$-bounded.
Finally, it seems that the induced minor relation has been studied the least with respect to $\chi$-boundedness. In view of this and the fact that \Cref{chi-bounded-induced-minor} identifies an infinite family of graphs $H$ such that the class of $H$-induced-minor-free graphs is $\chi$-bounded, we thus propose the following.

\begin{question}\label{question7}
For which graphs $H$ is the class of $H$-induced-minor-free graphs $\chi$-bounded?
\end{question}

Note that there are graphs $H$ such that the class of $H$-induced-minor-free graphs is not $\chi$-bounded.
For example, if $H$ is the graph obtained from the complete graph $K_5$ by subdividing each edge, then the class of $H$-induced-minor-free graphs contains the class of string graphs~\cite{MR2059098,sinden1966topology}, which is not $\chi$-bounded~\cite{MR3171778}.

\subsection*{Acknowledgements}

We are grateful to the anonymous reviewer for a careful reading of our paper and several helpful remarks, to Irena Penev for asking about implications of our results for $\chi$-bounded graph classes, to Jean-Florent Raymond for early discussions on treewidth of $1$-perfectly orientable graphs, and to Bart M.~P.~Jansen for asking a question that motivated~\cref{remark:K2n}.

\printbibliography
\end{document}